\documentclass[12pt]{amsart}
\usepackage{amsmath,amssymb,amsbsy,amsfonts,amsthm,latexsym,
                        amsopn,amstext,amsxtra,euscript,amscd,mathrsfs,color,bm,cite}
\usepackage{float}
\usepackage[english]{babel}
\usepackage{mathtools}
\usepackage{todonotes}
\usepackage{url}
\usepackage[colorlinks,linkcolor=blue,anchorcolor=blue,citecolor=blue,backref=page]{hyperref}
\usepackage{cases}
\usepackage{txfonts}
\usepackage{latexsym}
\usepackage{amsthm}
\usepackage{mathrsfs}
\usepackage{amssymb, amsmath}
\usepackage{cite}
\usepackage{color}
\usepackage{verbatim}

\def\re{\text{Re}}
\def\ma{\mathfrak{a}}
\def\mb{\mathfrak{b}}
\def\mc{\mathfrak{c}}
\def\r{\right}
\def\im{\text{Im}}

\def\msum{\mathop{\sum\nolimits^*}}

\def\abgd{\alpha,\beta,\gamma,\delta}
\def\bagd{\beta,\alpha,\gamma,\delta}
\def\abdg{\alpha,\beta,\delta,\gamma}
\def\badg{\beta,\alpha,\delta,\gamma}

\def\gdab{\gamma,\delta,\alpha,\beta}

\let\ve=\varepsilon

\let\ol=\overline

\let\vp=\varphi
\let\wt=\widetilde
\let\wh=\widehat
\let\as=\asymp
\theoremstyle{definition}
\newtheorem{definition}{Definition}[section]

\newtheorem*{remark}{Remark}
\newtheorem*{notation}{Notation}

\theoremstyle{plain}
\newtheorem{theorem}{Theorem}

\newtheorem{lemma}[theorem]{Lemma}
\newtheorem{conjecture}[theorem]{Conjecture}
\numberwithin{equation}{section}
\numberwithin{theorem}{section}
\newtheorem{proposition}[theorem]{Proposition}
\def\qed{\ifhmode\textqed\fi
   \ifmmode\ifinner\quad\qedsymbol\else\dispqed\fi\fi}
\def\textqed{\unskip\nobreak\penalty50
    \hskip2em\hbox{}\nobreak\hfil\qedsymbol
    \parfillskip=0pt \finalhyphendemerits=0}
\def\dispqed{\rlap{\qquad\qedsymbol}}

\begin{document}
\title{T\MakeLowercase{he fourth moment of }D\MakeLowercase{irichlet} $L$-\MakeLowercase{functions at the central value}}
\author{Xiaosheng Wu}
\date{}
\address {School of Mathematics, Hefei University of Technology, Hefei 230009,
P. R. China.}
\email {xswu@amss.ac.cn}

\begin{abstract}
The asymptotic formula of the fourth moment of Dirichlet $L$-functions at the central value was predicted in a conjecture by J. B. Conrey, D. W. Farmer, J. P. Keating, M. O. Rubinstein, and N. C. Snaith, and the prime moduli case was proved by M. P. Young in 2011.
This work establishes this asymptotic formula for general moduli.
The work relies on the study of a special divisor sum function, called $\mathcal{D}_q$-function, which plays a key role in deducing the main terms.
Another key ingredient is an application of new bounds for double sums in Kloosterman sums, applied in the estimate of the error terms.
\end{abstract}
\maketitle

\section{Introduction}
Moments of $L$-functions in families encode deep properties about the family and have a wide range of applications. Estimating moments of families of $L$-functions, especially their asymptotic formulae with a power saving error term, is regarded as a central problem in number theory. There is a very general conjecture on asymptotic formulae for moments of $L$-functions by Conrey, Farmer, Keating, Rubinstein, and Snaith \cite{CFK+05}. The most classical example is moments of the Riemann zeta-function, and the edge of current technology for this problem is the fourth moment. For other $L$-functions, the edge of current technology where one can hope to obtain an asymptotic formula with a power saving error term is in a similar complexity. For Dirichlet $L$-functions, it was a major breakthrough when Young \cite{You11} established an asymptotic formula with a power saving for the fourth moment at the central point for prime moduli.
For results of other $L$-functions in a similar complexity, one may refer to Kowalski, Michel, and VanderKam \cite{KMV00}, Iwaniec and Sarnak
\cite{IS00}, Blomer \cite{Blo04}, Li \cite{Li09}, Khan \cite{Kha12}, and Blomer and Mili\'cevi\'c \cite{BM15}.

A classical example is the fourth moment of the Riemann zeta-function, and it has been proved that
\begin{align}\label{eqZav}
\frac1T\int_0^T|\zeta(\tfrac12+it)|^4dt=P_4(\log T)+O\left(T^{-\frac13+\ve}\r),
\end{align}
for a certain polynomial $P_4$. For more detail, one may refer to Heath-Brown \cite{HB79}, Zavorotny\v{\i} \cite{Zav89}, Ivi\'c and Motohashi \cite{IM95}, and Motohashi\cite{Mot97}.

The fourth moment of Dirichlet $L$-functions in $q$-aspect is to some extent analogous to the Riemann zeta-function in $t$-aspect, but there exist significant differences. As specified in detail by Young \cite{You11}, these differences require quite different ways to treat the divisor sum and turn out to be much more difficult.

For $q\nequiv 2 \pmod4$, Heath-Brown \cite{HB81} proved
\[
\frac1{\varphi^*(q)}\msum_{\chi(\bmod q)}|L(\tfrac12,\chi)|^4=\frac1{2\pi^2}\prod_{p\mid q}\frac{(1-p^{-1})^3}{(1+p^{-1})}(\log q)^4+O\left(2^{\omega(q)}\frac{q}{\varphi^*(q)}(\log q)^3\r),
\]
where the sum is over all primitive characters modulo $q$, $\varphi^*(q)$ denotes the number of such primitive characters, and $\omega(q)$ denotes the number of distinct prime factors of $q$. Note that the condition $q\nequiv2 \pmod 4$ is reasonable since there are no primitive characters modulo $q$ if $q\equiv2 \pmod 4$. This asymptotic formula does not work when $q$ has too many distinct prime factors. This gap was then covered by Soundararajan \cite{Sou07} with a smaller error term, so that the first term of the asymptotic formula was justified for general moduli.

Due to a conjecture for integral moments in \cite{CFK+05}, it is believed that, for any integer $q\nequiv 2 \pmod4$,
\begin{align}\label{conjformula}
\frac1{\varphi^*(q)}\msum_{\chi(\bmod q)}|L(\tfrac12,\chi)|^4=\prod_{p\mid q}\frac{(1-p^{-1})^3}{(1+p^{-1})}P_4(\log q)+O\left(q^{-\frac12+\ve}\r)
\end{align}
with $P_4(x)$ being a computable absolute polynomial of degree $4$.

For prime moduli $p\ge3$, Young \cite{You11} pushed the result fairly close to the conjecture, more precisely, he proved
\begin{align}\label{Youngformula}
\frac1{\varphi^*(p)}\msum_{\chi(\bmod p)}|L(\tfrac12,\chi)|^4=P_4(\log p)+O\left(p^{-\frac1{80}+\frac{\theta}{40}+\ve}\r),
\end{align}
where $\theta$ denotes the exponent towards to the Ramanujan-Petersson conjecture, and one may take $\theta=7/64$, thanks to Kim and Sarnak \cite{Kim03}.

In more recent works by Blomer, Fouvry, Kowalski, Michel, and Mili\'cevi\'c \cite{BFK+17a,BFK+17b}, the error term of \eqref{Youngformula} has been sharpened to $p^{-\frac1{20}+\ve}$. In their works, two ingredients are applied to improve on the error,
one of which is some powerful results concerning bilinear forms in Kloosterman sums by Fouvry, Kowalski, and Michel \cite{FKM14},  Kowalski, Michel, and Sawin \cite{KMS17}, and Shparlinski and Zhang \cite{SZ16}, and the other is an average result concerning Hecke eigenvalues to remove the dependence on the Ramanujan-Petersson conjecture.

This work is devoted to the general moduli case for the conjecture of \cite{CFK+05}, where we shall prove that \eqref{conjformula} does hold for all $q$ with a power saving in the error term. We specify our main results in the following theorem.
\begin{theorem}\label{thmmain}
For any integer $q\nequiv 2 \pmod 4$, we have
\begin{align}\label{eqmain}
\frac1{\varphi^*(q)}\msum_{\chi(\bmod q)}|L(\tfrac12,\chi)|^4=\prod_{p\mid q}\frac{(1-p^{-1})^3}{(1+p^{-1})}P_4(\log q)+O\left(q^{-\frac1{14}+\frac37\theta+\ve}\r)
\end{align}
for a certain computable absolute polynomial $P_4(x)$.
\end{theorem}

\begin{remark}
With $\theta=7/64$, Theorem \ref{thmmain} gives an error term $q^{-\frac{11}{448}+\ve}$ for general moduli. Thanks to \cite{BFK+17a,BFK+17b}, the dependence on the Ramanujan-Petersson conjecture can be removed for prime moduli, and then Theorem \ref{thmmain} provides an error term $p^{-\frac1{14}+\ve}$, which also improves on
the error of  \cite{BFK+17a,BFK+17b} for prime moduli.\end{remark}

The problem will reduce to handling a divisor sum of type
\begin{align}\label{eqdivisor2}
\frac1{\varphi^*(q)}\sum_{d\mid q}\varphi(d)\mu(q/d)\mathop{\sum\sum}_{\substack{m\equiv n(\bmod d)\\ (mn,q)=1}}\frac{d(m)d(n)}{\sqrt{mn}}V\left(\frac{mn}{q^2}\r),
\end{align}
where $V(x)$ is a smooth function with a rapid decay on $x$, satisfying $V(x)\sim 1$ for small $x$. Diagonal terms with $m=n$ in \eqref{eqdivisor2} provide the first term of the polynomial in \eqref{eqmain}, which is not difficult to calculate.  To distinguish all lower-order main terms as well as a power saving in the error, we must perform complicated analysis on off-diagonal terms. We shall not take the space here to explain the difficulty of doing this, since one may find it in \cite{You11}.
Instead, we just introduce the new obstacles for the general moduli case, and give a sketch of our new ingredients here.

Let $B_{M,N}$ denote the same sum as in \eqref{eqdivisor2} but with $m, n$ restricted to the dyadic segments $M<m\le 2M$, $N<n\le 2N$. The treatment of $B_{M,N}$ proceeds essentially differently according to the relative size of $N$ with respect to $M$.

Our first obstacle is the coprime condition $(mn,q)=1$. When $q$ is prime, one may remove the coprime condition $(mn,q)=1$ with a negligible error, so can deduce the lower-order main terms by using the Estermann $D$-function. For composite moduli, the coprime condition will be reflected in the main terms, so one can not remove it any way. To handle the coprime condition, we should understand a special divisor sum function of type
\[
\mathcal{D}_{q}\left(s,\lambda,\frac h{l},r\r)=\sum_{\substack{(n,q)=1\\ (n+r,q)=1}}\frac{\sigma_\lambda(n)} {n^s} e\left(n\frac h{l}\r)
\]
with $\sigma_{\lambda}(n)=\sum_{d\mid n}d^\lambda$. The two coprime conditions make it hard to treat.
We shall deduce a functional equation for $\mathcal{D}_{q}$ as well as properties for its poles in Section \ref{secDq}, and
the deduction is elaborate and takes up much space. After applying the functional equation, we will distinguish the Kloosterman sums associated with a congruence group, and then utilize the Kuznetsov trace formula.

To bound the off-diagonal terms with $M$ and $N$ far away from each other such as $N/M>q^{1+\ve}$, we need to estimate the following double sums in Kloosterman sums
\[
\sum_{l\le L}\left|\msum_{k\le K} \alpha_k e\left(\frac{al\ol{k}}{q}\r)\r|,
\]
which we appeal to the Weil bound as well as a recent work by Kerr, Shparlinski, Wu, and Xi \cite{KSWX22}.

\subsection{The shifted fourth moment}
We have chosen to calculate a shifted fourth moment of Dirichlet $L$-functions including the shifts $\alpha,\beta,\gamma,\delta$, because the shifts can split a higher-order pole into several simple poles and thus allow for a clearer structure of the main terms. As mentioned by Young \cite{You11}, we may focus on the treatment of even characters, and the case of odd characters is similar (see also \cite[Section 8.3]{You11} for necessary changes).
Thus, what we treat directly is
\begin{align}
M(\alpha,\beta,\gamma,\delta)=\frac2{\vp^*(q)}\mathop{{\sum}^+}_{\chi(\bmod q)} L\left(\tfrac12+\alpha,\chi\r) L\left(\tfrac12+\beta,\chi\r) L\left(\tfrac12+\gamma,\ol{\chi}\r) L\left(\tfrac12+\delta,\ol{\chi}\r),\notag
\end{align}
where the symbol ``$+$" indicates that the summation is over all primitive even characters.

For convenience, we introduce some notations
\[
Z_q(\abgd)=\frac{\zeta_q(1+\alpha+\gamma)\zeta_q(1+\alpha+\delta)\zeta_q(1+\beta+\gamma)\zeta_q(1+\beta+\delta)}{\zeta_q(2+\alpha+\beta+\gamma+\delta)}
\]
and
\begin{align}\label{eqX}
X_{\alpha,\gamma}=X_{\alpha}X_{\gamma},\ \ \ \ \ \ \ \ X_{\alpha,\beta,\gamma,\delta}=X_{\alpha}X_{\beta}X_{\gamma}X_{\delta}
\end{align}
with
\begin{align}\label{eqdefX}
X_{\alpha}=\left(\frac{q}{\pi}\r)^{-\alpha}\frac{\Gamma\left(\frac{\frac12-\alpha+\ma}{2}\r)} {\Gamma\left(\frac{\frac12+\alpha+\ma}{2}\r)} \ \ \ \ \ \text{for}\ \ \ \ma=0,1.
\end{align}
Note that $Z_q(\abgd)$ is invariable under transpositions $\alpha\leftrightarrow\beta$ as well as $\gamma\leftrightarrow\delta$, and it is also invariable under the two transpositions $\alpha\leftrightarrow\gamma$, $\beta\leftrightarrow\delta$ simultaneously (e.g., $Z_q(-\gamma,-\delta,-\alpha,-\beta)=Z_q(-\alpha,-\beta,-\gamma,-\delta)$). In addition, $X_{\alpha,\gamma}$ and $X_{\alpha,\beta,\gamma,\delta}$ are invariable under any transpositions of the shifts.
\begin{conjecture}[CFK+05]\label{conjecture}
For any integer $q\nequiv 2 \pmod{4}$, and with shifts $\ll(\log q)^{-1}$, $\ma=0$, we have
\begin{align}\label{eqconjecture}
M(\alpha,\beta,\gamma,\delta)&=Z_{q}(\alpha,\beta,\gamma,\delta) +X_{\alpha,\beta,\gamma,\delta}Z_{q}(-\gamma,-\delta,-\alpha,-\beta)+X_{\alpha,\gamma}Z_{q}(\beta,-\gamma,\delta,-\alpha)\\
+&X_{\beta,\gamma}Z_{q}(\alpha,-\gamma,\delta,-\beta)
+X_{\alpha,\delta}Z_{q}(\beta,-\delta,\gamma,-\alpha) +X_{\beta,\delta}Z_{q}(\alpha,-\delta,\gamma,-\beta)+O\left(q^{-\frac12+\ve}\r).\notag
\end{align}

\end{conjecture}
There is a similar conjecture for the odd character case, and the only difference is to specify $\ma=1$ in \eqref{eqdefX}, while $\ma=0$ for the even character case.

\begin{theorem}\label{thmmains}
Conjecture \ref{conjecture} holds but with an error of size $O\left(q^{-\frac1{14}+\frac37\theta+\ve}\r)$.
\end{theorem}

The symmetry implies that the right-hand side of \eqref{eqconjecture} is holomorphic concerning the shifts, which has been proved in a more general setting in \cite[Lemma 2.5.5]{CFK+05}. Thus, taking the limit as the shifts go to zero in \eqref{eqconjecture} gives Theorem \ref{thmmain}.

Note that $M(\alpha,\beta,\gamma,\delta)$ and all the main terms in the right-hand side of \eqref{eqconjecture} are holomorphic concerning the shifts, and so is the error. Thus, the maximum modulus principle implies that it suffices to prove Theorem \ref{thmmains} for such special shifts that,
each of the shifts lies in a fixed annulus with inner and outer radii $\asymp (\log q)^{-1}$, which are separated enough so that $|\alpha\pm\beta|\gg(\log q)^{-1}$, etc.

\subsection{Notation}
Some notations should be presented here for ease of reference.
As usual, we use $\ve$ to denote an arbitrarily small positive constant that may vary from line to line, and $e(x)$ means the exponential function $e^{2\pi ix}$. $\varphi(n)$ is the Euler function, which denotes the number of positive integers less than $n$ and coprime with $n$. The notations $\sigma_\lambda(n)$ and $\sigma_{\alpha,\beta}(n)$ are defined as
\[
\sigma_\lambda(n)=\sum_{d\mid n}d^\lambda,\ \ \ \ \sigma_{\alpha,\beta}(n)=\sum_{ad=n}a^{-\alpha} d^{-\beta},
\]
and it is obvious that
$\sigma_{\alpha,\beta}(n)=n^{-\alpha}\sigma_{\alpha-\beta}(n)$.

The symbol $\theta$ always denotes an admissible exponent for the Ramanujan-Petersson conjecture for Maass newforms, and the current best-known result is $\theta=7/64$.

In this paper, $W(x)$ always denotes a smooth non-negative function compactly supported in $[1,2]$, which may have different expressions for each occurrence, and we apply $\wt{W}(u)$ and $\wh{W}(u)$ to denote its Mellin transform and Fourier transform respectively.

For notational convenience, we apply two symbols related to $q$. We use $q^*$ to denote the largest square-free divisor of $q$, and $q_k$ to be the largest divisor of $q$ coprime with $k$. Note that $q_0, q_1, q_2, \dots$ are not the case, where the subscripts are just serial numbers.

In the following, functions we face are often on the shifts $\alpha,\beta,\gamma,\delta$, such as $M(\alpha,\beta,\gamma,\delta)$. In the absence of ambiguity, we frequently omit the shifts for convenience when they are in the order of $\alpha,\beta,\gamma,\delta$.

We apply two special summation notations
\[
\msum_{n (\bmod q)}\ \ \ \ \text{and} \ \ \ \ \msum_{\chi (\bmod q)},
\]
where the first is a sum over such $n$ coprime with $q$, and the second is over all primitive characters modulo $q$.

\section{Background and auxiliary lemmas}
\subsection{The Dirichlet $L$-function and the Hurwitz zeta-function}
Let $q$ be a positive integer, and $\chi$ be a Dirichlet character modulo $q$. For $\re(s)>1$, the Dirichlet $L$-function may be defined as
$L(s,\chi)=\sum_{n=1}^\infty\chi(n)n^{-s}$.
It satisfies the functional equation
\[
\left(\frac q\pi\r)^{\frac s2}\Gamma\left(\frac{\frac12+s+\ma}{2}\r)L\left(\tfrac12+s,\chi\r)=i^{-\ma}q^{-\frac12}\tau(\chi)\left(\frac q\pi\r)^{-\frac s2}\Gamma\left(\frac{\frac12-s+\ma}{2}\r) L\left(\tfrac12-s,\ol{\chi}\r),
\]
where
\begin{align}
\tau(\chi)=\sum_{x(\bmod q)}\chi(x)e\left(\frac{x}q\r),\ \ \ \ \ \ \ \ma=\left\{
\begin{aligned}
&0, \ \ \ \ \chi(-1)=1,\\
&1, \ \ \ \ \chi(-1)=-1.\notag
\end{aligned}
\right.
\end{align}
For $\re(s)>1$ and $0<x\le1$, the Hurwitz zeta-function is defined as
$\zeta(s,x)=\sum_{n=0}^\infty(n+x)^{-s}$,
and it has the functional equation
\[
\zeta(1-s,x)=(2\pi)^{-s}\Gamma(s)\left(e\left(-\frac{\pi s}{2}\r)F(s,x)+e\left(\frac{\pi s}{2}\r)F(s,-x)\r)
\]
with
$F(s,x)=\sum_{n=1}^{\infty} e(nx)n^{-s}$.

\subsection{Approximate functional equation}
To expand $M(\alpha,\beta,\gamma,\delta)$, we apply the following approximate functional equation, deduced from the functional equation of $L(s,\chi)$ standardly; see also \cite[Proposition 2.4]{You11}.
\begin{lemma}[Approximate functional equation]
\label{lemafe}
Let $G(s)$ be an even, entire function of exponential decay in any strip $|\re(s)|<C$ satisfying $G(0)=1$, and let
\begin{align}
\label{defV}
V_{\alpha,\beta,\gamma,\delta}(x)=\frac1{2\pi i}\int_{(1)}\frac{G(s)}{s}g_{\alpha,\beta,\gamma,\delta}(s)x^{-s}ds,
\end{align}
where
\begin{align}
\label{defg}
g_{\alpha,\beta,\gamma,\delta}(s)=\pi^{-2s}\frac{\Gamma\left(\frac{ \frac12+\alpha+s+\ma}2\r) \Gamma\left(\frac{ \frac12+\beta+s+\ma}2\r) \Gamma\left(\frac{ \frac12+\gamma+s+\ma}2\r) \Gamma\left(\frac{ \frac12+\delta+s+\ma}2\r)}{\Gamma\left(\frac{\frac12+\alpha+\ma}2\r) \Gamma\left(\frac{\frac12+\beta+\ma}2\r) \Gamma\left(\frac{\frac12+\gamma+\ma}2\r) \Gamma\left(\frac{\frac12+\delta+\ma}2\r)}.
\end{align}
Furthermore, set
\begin{align}
\widetilde{V}_{\alpha,\beta,\gamma,\delta}(x)= X_{-\gamma,-\delta,-\alpha,-\beta}V_{\alpha,\beta,\gamma,\delta}(x)\notag
\end{align}
with $X_{*,*,*,*}$ defined by \eqref{eqX}.
Then
\begin{align}
L\left(\tfrac12+\alpha,\chi\r) &L\left(\tfrac12+\beta,\chi\r) L\left(\tfrac12+\gamma,\ol{\chi}\r) L\left(\tfrac12+\delta,\ol{\chi}\r)\notag\\
&=\sum_{m,n}\frac{\sigma_{\alpha,\beta}(m)\sigma_{\gamma,\delta}(n) \chi(m)\ol{\chi}(n)} {(mn)^{\frac12}}V_{\alpha,\beta,\gamma,\delta}\left(\frac{mn}{q^2}\r)\notag\\ &\ \ \ + \sum_{m,n}\frac{\sigma_{-\gamma,-\delta}(m)\sigma_{-\alpha,-\beta}(n) \chi(m)\ol{\chi}(n)} {(mn)^{\frac12}}\wt{V}_{-\gamma,-\delta,-\alpha,-\beta}\left(\frac{mn}{q^2}\r).\notag
\end{align}
\end{lemma}

This approximate functional equation holds for a general $G$.  However, we will apply it with a special $G$ in the following.
\begin{definition}[Definition of $G(s)$]\label{DefG}
Let $G(s)=P_{\alpha,\beta,\gamma,\delta}(s)\exp(s^2)$, where $P_{\alpha,\beta,\gamma,\delta}(s)$ is an even polynomial in $s$, satisfying some properties: it is rational, symmetric, and even in the shifts; $P_{\alpha,\beta,\gamma,\delta}(0)=1$, $P_{\alpha,\beta,\gamma,\delta}\left(-\frac{\alpha+\gamma}2\r)=0$, and $P_{\alpha,\beta,\gamma,\delta}\left(\frac12\pm\alpha\r)=0$ (as well as $\frac{\beta+\delta}2, \frac12\pm\beta$, etc., by symmetry).
\end{definition}

\subsection{The orthogonality formula}
We apply the following orthogonality formula to execute the sum over characters, whose proof may be referred to many sources, such as \cite{HB81} and \cite{Sou07}.
\begin{lemma}[The orthogonality formula]
\label{lemof}
For $(mn,q)=1$, we have
\begin{align}
\label{lemof2}
\mathop{{\sum}^+}_{\chi(\bmod q)}\chi(m)\ol{\chi}(n) =\frac12\left(\sum_{d\mid(q,m-n)}\vp(d)\mu\left(\frac qd\r) +\sum_{d\mid(q,m+n)}\vp(d)\mu\left(\frac qd\r)\r).
\end{align}
\end{lemma}

\subsection{Additive characters}
Frequently, we write the divisibility condition as well as the coprime condition in terms of additive characters, which are presented here for ease of reference.
\begin{lemma}\label{lemcoprimec}
Let $a(n)$ be an arithmetic function. For any integer $k$, we have
\begin{align}\label{eqdivisor}
\sum_{k\mid n}a(n)=\frac1k\sum_{d\mid k}\sum_nc_{d}(n)a(n),
\end{align}
where $c_d(n)$ is the Ramanujan sum.
\end{lemma}
\begin{proof}
Applying the sum of divisors formula to the Ramanujan sum and reversing the order of the summation, we have
\begin{align}
\sum_{d\mid k}c_{d}(n)=\sum_{d\mid k}\sum_{l\mid (d,n)}l\mu\left(\frac {d}{l}\r)
=\sum_{l\mid (k,n)}l\sum_{m\mid \frac kl}\mu(m)=\left\{
\begin{aligned}
&k \ \ \ \text{for}\ \ \ k\mid n,\\
&0 \ \ \ \text{otherwise}.\notag
\end{aligned}
\right.
\end{align}
Then \eqref{eqdivisor} follows immediately.
\end{proof}

\begin{lemma}
Let $a(n)$ be an arithmetic function. For any integer $q$, we have
\begin{align}\label{eqcoprimee}
\sum_{(n,q)=1}a(n)=\sum_{d\mid q}\frac{\mu(d)}{d}\sum_{r (\bmod d)} \sum_n e\left(\frac{rn}{d}\r)a(n)
\end{align}
and
\begin{align}\label{eqcoprime}
\sum_{(n,q)=1}a(n)=\sum_{d\mid q}\frac{\mu(d)}{d}\sum_{d_1\mid d}\sum_nc_{d_1}(n)a(n).
\end{align}
\end{lemma}

It is an exercise to check \eqref{eqcoprimee} directly, as well as \eqref{eqcoprime} with \eqref{eqdivisor}.

\subsection{Multiplicative characters}
Sometimes, especially in the deduction of a functional equation for the $\mathcal{D}_q$-function (see Section \ref{sec5.2}, Proposition \ref{lemgED2}), we need to write the coprime condition in terms of multiplicative characters. We present this in the following lemma, and one may check it directly.
\begin{lemma}\label{lemcoprimechie}
Let $a(n)$ be an arithmetic function. For any integers $r$ and $q$, we have
\begin{align}\label{eqcoprimechi}
\sum_{\substack{(n,q)=1\\ (n+r,q)=1}}a(n)=\sum_{b\mid q}\frac{\mu(b)}{\vp(b)} \sum_{\chi(\bmod b)}\ol{\chi}(-r)\sum_{n}\chi(n)\chi_0(n)a(n),
\end{align}
where $\chi_0$ is the principal character modulo $q$.
\end{lemma}

\subsection{The reciprocity law}
To split and combine the exponential function, we will frequently apply the reciprocity law, a direct result of the Chinese Remainder Theorem.
\begin{lemma}\label{lemmaspe(x)}
Let $a,b$ be two coprime integers.
We have
\begin{align}
e\left(\frac x{ab}\r)=e\Bigg(\frac{y\ol{b}}{a}\Bigg)e\left(\frac{z\ol{a}}{b}\r),\notag
\end{align}
if $x\equiv y \pmod a$ and $x\equiv z \pmod b$.
\end{lemma}

\subsection{An arithmetic formula}
In some ways, the following arithmetic formula may reflect the cancelation of M\"{o}bius function in \eqref{eqdivisor2}.
\begin{lemma}\label{lemsumd}
For any integer $q$, we have
\begin{align}
\sum_{d\mid q}\frac{\varphi(d)}{d}\mu\left(\frac qd\r)=\frac{\mu(q)}q.\notag
\end{align}
\end{lemma}
\begin{proof}
Note that both sides of the formula are multiplicative functions on $q$, so it is sufficient to check the identity for prime powers, which is an exercise.
\end{proof}

\subsection{A dyadic partition of unity}
We also require a dyadic partition of unity. Let $W$ be a smooth non-negative function compactly supported in $[1,2]$ such that, for any $x\in\mathbb{R}^+$,
\begin{align}
\sum_{M}W\left(\frac xM\r)=1.\notag
\end{align}
Here $M$ varies over a set of positive real numbers, with the number of such $M$ less than $X$ being $O(\log X)$.
By taking products of the $W\left(\frac xM\r)$, we then create the partition of unity $\left\{W\left(\frac xM\r)W\left(\frac yN\r)\r\}$ on $\mathbb{R}^+\times\mathbb{R}^+$.

\subsection{Automorphic preliminaries}
We sketch some elementary materials we require on automorphic forms to apply the Kuznetsov formula in this section, which mainly follows from \cite{BM15}, \cite{Iwa95}, and \cite{DI82}. We write the Fourier expansion of a holomorphic modular form $f$ of level $Q$ and weight $k$ as
\[
f(\sigma_{\ma}z)=\sum_{n\ge1}\rho_f(\ma,n)(4\pi n)^{\frac k2}e(nz),
\]
and similarly write for a Maass  form $f$ of level $Q$ and spectral parameter $\kappa=\kappa_f\in\mathbb{R}\cup[-i\theta,i\theta]$ as
\[
f(\sigma_{\ma}z)=\sum_{n\neq0}\rho_f(\ma,n)W_{0,i\kappa}(4\pi|n|y)e(nx),
\]
where $W_{0,i\kappa}(y)=(y/\pi)^{1/2}K_{i\kappa}(y/2)$ is a Whittaker function. For cusps $\ma, \mc$ of $\Gamma_0(Q)$, the Fourier expansion of the Eisenstein series $E_{\mc}(\sigma_\ma z,s)$ at $s=\frac12+i\kappa$ is written as
\[
E_{\mc}(\sigma_\ma z, \tfrac12+i\kappa)=\delta_{\ma,\mc}y^{1/2+i\kappa}+\varphi_{\ma,\mc}(\tfrac12+i\kappa)y^{1/2-i\kappa} +\sum_{n\neq0}\rho_{\ma,\mc}(n,\kappa)W_{0,i\kappa}(4\pi|n|y)e(nx).
\]
For the special case $\ma=\infty$, we denote the Fourier coefficients as
\[
\rho_f(n)=\rho_f(\infty,n), \ \ \ \ \ \ \ \ \ \ \ \ \rho_{\mc}(n,\kappa)=\rho_{\infty,\mc}(n,\kappa)
\]
respectively.
If $f$ is a cuspidal newform with normalized Hecke eigenvalues denoted by $\lambda_f(n)$, there is
$\lambda_{f}(n)\rho_f(1)=\sqrt{n}\rho_f(n)$.
It is generally believed that
$|\lambda_f(n)|\le d(n)$,
which is known as the Ramanujan-Petersson conjecture, and the best result known now is
$\lambda_f(n)\ll n^{\theta+\ve}$
with $\theta=7/64$, due to Kim and Sarnak \cite{Kim03}.

Not all cusp forms are newforms, but the newform theory allows an orthogonal basis based on normalized newforms.
Taking Maass forms as an example,
one may find a special orthogonal basis $\mathcal{B}(Q)$ that, for each $f\in \mathcal{B}(Q)$, there is a normalized newform $f^*$ of level $Q_1$ with $Q_1\mid Q$ and of the same spectral parameter (see also \cite[Section 5]{BM15}). Let $\mathcal{B}(Q)$ and $\mathcal{B}_k(Q)$ be two $L^2$-basises in this case, for the spaces of Maass forms of level $Q$ and holomorphic cusp forms of weight $k$ and level $Q$ respectively.

We will need to split variables in the Fourier coefficients, which usually follows from a multiplicative property. The Fourier coefficients of the forms in the bases $\mathcal{B}_k(Q)$ and $\mathcal{B}(Q)$ are not exactly multiplicative, but are almost so (see also \cite{BHM07}, \cite{BM15}).
In particular,
if $m=bm'$ with $(m',b)=1$, then
\begin{align}\label{eqMrho}
\sqrt{m}\rho_f(m)=\sum_{d\mid (Q,b/(b,Q))}\mu(d)\chi_0(d)\lambda_{f^*}\left(\frac{b}{d(b,Q)}\r)\left(\frac{(b,Q)m'}{d}\r)^{\frac12}\rho_f\left(\frac{(b,Q)m'}{d}\r),
\end{align}
where $f^*$ is the underlying newform and $\chi_0$ is the trivial character modulo $Q_1$ (see also \cite[formula (5.2)]{BM15}).
Moreover, if $f$ is a holomorphic form, $f^*$ will satisfy the Ramanujan-Petersson conjecture. If $b\in\mathbb{N}$ and $a_m$ is any finite sequence of complex numbers supported on the integers $m=bm'$ with $(m',b)=1$, then
\begin{align}\label{eqhrho}
\left|\sum_m a_m\sqrt{m}\rho_f(m)\r|^2\le d(b)^2\sum_{d\mid(b,Q)}\left|\sum_{m'}a_{bm'}\sqrt{dm'}\rho_f(dm')\r|^2
\end{align}
(see also \cite[formula (5.4)]{BM15}).
There is a similar result for the coefficients of Eisenstein series as
\begin{align}\label{eqErho}
\sum_{\mathfrak{c}}\bigg|\sum_m a_m\sqrt{m}\rho_{\mathfrak{c}}(m,\kappa)\bigg|^2\le 9d(Q)^3d(b)^4\sum_{d\mid(b,Q)}\sum_{\mathfrak{c}}\bigg|\sum_{m'} a_{bm'}\sqrt{dm'}\rho_{\mathfrak{c}}(dm',\kappa)\bigg|^2
\end{align}
(see also \cite[ formula (5.5)]{BM15}).

\subsection{Kloosterman sums for congruence groups and the Kuznetsov formula}
In this section, we present some materials we require on Kloosterman sums for congruence groups, which can be referred to \cite{DI82}.
The classical Kloosterman sum is defined as
\begin{align}\label{eqdefKSC}
S(m,n;q)=\msum_{d(\bmod q)}e\Bigg(\frac{m\ol{d}+nd}{q}\Bigg).
\end{align}
Let $\ma$ and $\mb$ be two cusps of the Hecke congruence group $\Gamma_0(Q)$ and let $\sigma_{\ma}$, $\sigma_{\mb}$ in $PSL(2,\mathbb{R})$ be scaling matrixes. For any $m,n\in\mathbb{Z}$ and $\gammaup\in\mathbb{R}$, for which there exists
$\begin{pmatrix}
     \alphaup & \betaup \\
     \gammaup & \deltaup
\end{pmatrix}$
in $\sigma_{\ma}^{-1}\Gamma_0(Q)\sigma_{\mb}$, the Kloosterman sum is defined by
\begin{align}\label{eqdefS}
S_{\ma,\mb}(m,n;\gamma)=\mathop{\sum\nolimits^{'}}_{\deltaup(\bmod {\gammaup\mathbb{Z}})}e\left(m\frac{\alphaup}{\gammaup}+n\frac{\deltaup}{\gammaup}\r),
\end{align}
where the sum is over the $\deltaup$'s, taken modulo $\gamma\mathbb{Z}$, for which there exist $\alpha$ and $\beta$ such that $\begin{pmatrix}
     \alphaup & \betaup \\
     \gammaup & \deltaup
\end{pmatrix}$
in $\sigma_{\ma}^{-1}\Gamma_0(Q)\sigma_{\mb}$.

An interesting example introduced in \cite{DI82} is the particular case we focus on. Let $Q=\tau s$ with $(\tau,s)=1$, and we consider $S_{\infty,1/s}(m,n;\gammaup)$ with $\sigma_{1/s}=\begin{pmatrix}
     \sqrt{\tau} & 0 \\
     s\sqrt{\tau} & \frac1{\sqrt{\tau}}
\end{pmatrix}$.
From the definition \eqref{eqdefS}, one notes that
$S_{\infty,1/s}(m,n;\gammaup)$ is defined if and only if, $\gammaup$ may be written as $\gammaup=s\sqrt{\tau}C$ with integer $C$ coprime with $\tau$, and this means
\begin{align}\label{eqSS}
S_{\infty,1/s}(m,n;\gammaup)=e\left(n\frac{\ol{s}}{\tau}\r)S(m\ol{\tau},n;sC),
\end{align}
where $s\ol{s}\equiv1 \pmod{\tau}$ and $\tau\ol{\tau}=1 \pmod {sC}$
(see also \cite[formula (1.6)]{DI82}).

To estimate a sum of Kloosterman sums of type \eqref{eqdefS}, we require the following Kuznetsov trace formula (see also \cite[Theorems 9.4, 9.5, 9.7]{Iwa95}).
\begin{lemma}[Kuznetsov formula]
Let $m,n$ be two positive integers and $\phi$ be a function $\mathcal{C}^2$ class on $[0,\infty)$, satisfying $\phi(0)=0$, $\phi^{(j)}(x)\ll(1+x)^{-2-\ve}$ for $j=0,1,2$; let $\ma$ and $\mb$ be two cusps of  $\Gamma=\Gamma_0(Q)$; denoting by $\sum\limits^{\Gamma}$ a summation performed over the positive real numbers $\gammaup$ for which $S_{\ma,\mb}(m,n;\gammaup)$ is defined, one has
\begin{align}
\sum^{\Gamma}\frac1{\gammaup}S_{\ma,\mb}(m,n;\gammaup)\phi\left(\frac{4\pi\sqrt{mn}}{\gammaup}\r)=&\sum_{2\le k\equiv0(\bmod 2)}\sum_{f\in \mathcal{B}_k(Q)}\Gamma(k)\tilde{\phi}(k)\sqrt{mn}\ol{\rho}_f(\ma,m)\rho_f(\mb,n)\notag\\
&+\sum_{f\in\mathcal{B}(Q)}\hat{\phi}(\kappa_f)\frac{\sqrt{mn}}{\cosh(\pi\kappa_f)}\ol{\rho}_f(\ma,m)\rho_f(b,n)\notag\\
&+\frac1{4\pi}\sum_{\mc}\int_{-\infty}^{\infty}\hat{\phi}(\kappa)\frac{\sqrt{mn}}{\cosh(\pi\kappa)}\ol{\rho}_{\ma,\mc}(m,\kappa) \rho_{\mb,\mc}(n,\kappa) d\kappa\notag
\end{align}
and
\begin{align}
\sum^{\Gamma}\frac1{\gammaup}S_{\ma,\mb}(m,-n;\gammaup)\phi\left(\frac{4\pi\sqrt{mn}}{\gammaup}\r)= &\sum_{f\in\mathcal{B}(Q)}\breve{\phi}(\kappa_f)\frac{\sqrt{mn}}{\cosh(\pi\kappa_f)}\ol{\rho}_f(\ma,m)\rho_f(b,-n)\notag\\
&+\frac1{4\pi}\sum_{\mc}\int_{-\infty}^{\infty}\breve{\phi}(\kappa)\frac{\sqrt{mn}}{\cosh(\pi\kappa)}\ol{\rho}_{\ma,\mc}(m,\kappa) \rho_{\mb,\mc}(-n,\kappa) d\kappa,\notag
\end{align}
where the Bessel transforms are defined by
\begin{align}
&\tilde{\phi}(k)=4i^k\int_0^\infty\phi(x)J_{k-1}(x)\frac{dx}x,\notag\\
&\hat{\phi}(\kappa)=2\pi i\int_0^\infty\phi(x)\frac{J_{2i\kappa}(x)-J_{-2i\kappa}(x)}{\sinh(\pi \kappa)}\frac{dx}x,\notag\\
&\breve{\phi}(\kappa)=8\int_0^\infty\phi(x)\cosh(\pi\kappa)K_{2i\kappa}(x)\frac{dx}x.\notag
\end{align}
\end{lemma}

The Kuznetsov formula is often used together with the spectral large sieve inequalities, proved by Deshouillers and Iwaniec \cite[Theorem 2]{DI82}.
\begin{lemma}[Spectral large sieve]\label{lemsls}
Let $K\ge 1$, $N\ge1$, and $(a_n)$ be a sequence of complex numbers. Let $\ma$ be a cusp of $\Gamma_0(Q)$, which is equivalent to some $\frac u w$ with positive coprime $u$ and $w$ such that $w\mid Q$. Then all three quantities
\[
\sum_{\substack{2\le k\le K\\ k~ \text{even}}}\Gamma(k)\sum_{f\in\mathcal{B}_k(Q)}\left|\sum_{n}a(n)\sqrt{n}\rho_f(\ma,n)\r|^2,  \ \ \ \ \ \ \ \ \ \ \ \ \sum_{\substack{f\in\mathcal{B}(Q)\\ |\kappa_f|\le K}}\frac1{\cosh(\pi\kappa_f)}\left|\sum_n a_n\sqrt{n}\rho_f(\ma,\pm n)\r|^2,
\]
\[
\sum_{\mc}\int_{-K}^K\frac1{\cosh(\pi\kappa)}\left|\sum_n a_n\sqrt{n}\rho_{\ma,\mc}(\pm n,\kappa)\r|^2d\kappa
\]
are bounded by
\[
(K^2+\muup(\ma)N^{1+\ve})\sum_n|a_n|^2,
\]
where $\muup(\ma)$ is defined as
$\muup(\infty)=Q^{-1}$ and $\muup(\ma)=\left(w,\frac Qw\r)Q^{-1}$.
\end{lemma}

\section{Sketch of the proof}\label{secT}
We give a quick outline of the proof of Theorem \ref{thmmains} in this section.
After applying the approximate functional equation of Lemma \ref{lemafe} and the orthogonality formula of Lemma \ref{lemof}, we split $M$ as
\begin{align}
M(\alpha,\beta,\gamma,\delta)=A_{1}(\alpha,\beta,\gamma,\delta) +A_{-1}(\alpha,\beta,\gamma,\delta),\notag
\end{align}
where
\begin{align}
A_1(\alpha,\beta,\gamma,\delta)=&\frac1{\vp^*(q)}\sum_{d\mid q}\vp(d)\mu\left(\frac qd\r)\sum_{\substack{(mn,q)=1\\ m\equiv \pm n(\bmod d)}}\frac{\sigma_{\alpha-\beta}(m) \sigma_{\gamma-\delta}(n)} {m^{\frac12+\alpha}n^{\frac12+\gamma}}
V_{\alpha,\beta,\gamma,\delta}\left(\frac{mn}{q^2}\r)\notag
\end{align}
and
\begin{align}\label{eqrA_1A_{-1}}
A_{-1}(\alpha,\beta,\gamma,\delta)=X_{\abgd}A_1(-\gamma,-\delta,-\alpha,-\beta).
\end{align}
Thus, it is sufficient to deal with $A_1$.

To treat $A_1$, we split it into diagonal terms and off-diagonal terms, and then classify the off-diagonal terms into two categories according to $m\equiv n \pmod d$ and $m\equiv -n\pmod d$. Furthermore, we split $m\equiv n \pmod d$ according to $m<n$ and $m>n$. Consequently, we rewrite $A_{1}$ as
\begin{align}\label{eqA_1}
A_{1}(\alpha,\beta,\gamma,\delta)=A_{D}(\alpha,\beta,\gamma,\delta) +A_{O}(\alpha,\beta,\gamma,\delta)+A^*_{O}(\alpha,\beta,\gamma,\delta) +A_{\ol{O}}(\alpha,\beta,\gamma,\delta),
\end{align}
where $A_{D}$ and $A_{\ol{O}}$ denote respective contributions from the diagonal terms and off-diagonal terms with $m\equiv -n \pmod d$, and where $A_{O}$ and $A^*_{O}$ denote contributions from such off-diagonal terms with $m\equiv n \pmod d$, according to $m<n$ and $m>n$.

The diagonal terms in $A_D$ mainly contribute to the first term of the asymptotic formula, which is not complicated and would be obtained in the next section.
For the off-diagonal terms, there is an obvious relationship that
\begin{align}\label{A*OAO}
A^*_O(\abgd)=A_O(\gdab),
\end{align}
and thus it is sufficient to treat $A_O$ and $A_{\ol{O}}$.

Applying the dyadic partition of unity to the sums over $m, n$, we have
\begin{align}\label{eqAOBMN}
A_{O}=\sum_{M, N} B_{M,N},\ \ \ \ \ \ A_{\ol{O}}=\sum_{M\le N} B_{\ol{M,N}}+\sum_{M>N} B_{\ol{M,N}},
\end{align}
where
\begin{align}
B_{M,N}=\frac1{\vp^*(q)}\sum_{d\mid q}\vp(d)\mu\left(\frac qd\r) \sum_{\substack{(mn,q)=1, m<n\\ m\equiv  n(\bmod d)}}\frac{\sigma_{\alpha-\beta}(m) \sigma_{\gamma-\delta}(n)} {m^{\frac12+\alpha}n^{\frac12+\gamma}}
V_{\alpha,\beta,\gamma,\delta}\left(\frac{mn}{q^2}\r)W\left(\frac mM\r) W\left(\frac nN\r)\notag
\end{align}
and
\begin{align}
B_{\ol{M,N}}=\frac1{\vp^*(q)}\sum_{d\mid q}\vp(d)\mu\left(\frac qd\r) \sum_{\substack{(mn,q)=1\\ m\equiv  -n(\bmod d)}}\frac{\sigma_{\alpha-\beta}(m) \sigma_{\gamma-\delta}(n)} {m^{\frac12+\alpha}n^{\frac12+\gamma}}
V_{\alpha,\beta,\gamma,\delta}\left(\frac{mn}{q^2}\r)W\left(\frac mM\r) W\left(\frac nN\r).\notag
\end{align}
By swapping $M$ and $N$, one notes that
\begin{align}\label{BBMN}
B_{\ol{M,N}}(\abgd)=B_{\ol{N,M}}(\gdab),
\end{align}
which means the treatment of $B_{\ol{M,N}}$ for $M\le N$ is enough.
Since our focus is on $B_{\ol{M,N}}$ with $M\le N$ and $B_{M,N}$ with $m<n$, we may assume
\[
M\ll Nq^\ve,
\]
a convention that holds true throughout the work. We may also assume $MN\le q^{2+\ve}$ due to the rapid decay of $V$.

To evaluate $B_{M,N}$ and $B_{\ol{M,N}}$, we write them as
\begin{align}\label{+ME}
B_{M,N}=(\text{Main term})_{M,N}+E_{M,N},\ \ \ \ \ \ \ B_{\ol{M,N}}=(\text{Main term})_{\ol{M,N}}+E_{\ol{M,N}}
\end{align}
for certain main terms that are too complex to explicitly write here.
We deduce exact expressions for the main terms in Section \ref{secmainAO}. Deduced from the exact expressions, the following upper bound
\begin{align}\label{eqbmain}
(\text{Main term})_{M,N},~ (\text{Main term})_{\ol{M,N}}\ll M^{\frac12}N^{-\frac12}q^\ve
\end{align}
holds, which turns out to be very small when $M,N$ are not close enough.

We deduce an estimate for error terms in Section \ref{secError}.
\begin{theorem}\label{thmEMN}
For $M\ll Nq^\ve $, $MN\ll q^{2+\ve}$, we have
\begin{align}\label{eqEMN}
E_{M,N}, ~E_{\ol{M,N}}\ll q^{-\frac12+\theta+\ve}M^{-\frac12}N^{\frac12}.
\end{align}
\end{theorem}

In Section \ref{secmain}, we combine all various main terms from $A_O$, $A_{\ol{O}}$, etc., to deduce the main terms of Theorem \ref{thmmains}. Note that the bound \eqref{eqEMN} does not work well when $M$ and $N$ are far away from each other. So we complete our proof of Theorem \ref{thmmains} in Section \ref{secfar} by bounding the off-diagonal terms with $M,N$ far away from each other, where an estimate of a double sum in Kloosterman sums is applied. We will also need the trivial bound
\begin{align}\label{eqtrivalbound}
B_{M,N}, B_{\ol{M,N}}\ll q^{-1+\ve}(MN)^{\frac12}.
\end{align}

\section{Diagonal terms}
For the diagonal terms with $m=n$, it is
\begin{align}
A_D(\alpha,\beta,\gamma,\delta)=&\frac1{\vp^*(q)}\sum_{d\mid q}\vp(d)\mu\left(\frac qd\r)\sum_{(n,q)=1}\frac{\sigma_{\alpha-\beta}(n) \sigma_{\gamma-\delta}(n)} {n^{1+\alpha+\gamma}}
V_{\alpha,\beta,\gamma,\delta}\left(\frac{n^2}{q^2}\r)\notag\\
=&\frac1{2\pi i}\int_{(1)}\frac{G(s)}{s}g_{\alpha,\beta,\gamma,\delta}(s)q^{2s} \sum_{(n,q)=1}\frac{\sigma_{\alpha-\beta}(n) \sigma_{\gamma-\delta}(n)} {n^{1+\alpha+\gamma+2s}}ds.\notag
\end{align}
By the Ramanujan identity, the sum over $n$ is
\begin{align}
\frac{\zeta_q(1+\alpha+\gamma+2s) \zeta_q(1+\alpha+\delta+2s)\zeta_q(1+\beta+\gamma+2s)\zeta_q(1+\beta+\delta+2s)} {\zeta_q(2+\alpha+\beta+\gamma+\delta+4s)},\notag
\end{align}
which has simple poles at $2s=-\alpha-\gamma$, etc., while $G(s)$ vanishes at these poles. We move $\re(s)$ to $-\frac14+\ve$, passing a pole at $s=0$ only. The integral along the new line is
bounded by
$\ll q^{-\frac12+\ve}$,
and the pole at $s=0$ gives
$Z_q(\alpha,\beta,\gamma,\delta)$.
We summarize this in the following:
\begin{lemma}
\label{lemD}
We have
\begin{align}
A_D(\alpha,\beta,\gamma,\delta)=Z_q(\alpha,\beta,\gamma,\delta) +O\left(q^{-\frac12+\ve}\r),\notag
\end{align}
and similarly the contribution of the diagonal terms to $A_{-1}$ is
\begin{align}
A_{-D}(\alpha,\beta,\gamma,\delta)=X_{\alpha,\beta,\gamma,\delta} Z_q(-\gamma,-\delta,-\alpha,-\beta) +O\left(q^{-\frac12+\ve}\r).\notag
\end{align}
\end{lemma}

\section{The $\mathcal{D}_{q}$-function}\label{secDq}
To deduce an asymptotic formula for the off-diagonal terms, we need to treat a divisor sum of type
\[
\sum_{n\in S_q}\frac{\sigma_\lambda(n)} {n^s} e\left(n\frac h{l}\r)
\]
with $S_q$ being a set of positive integers meeting some coprime conditions, such as $S_q=\{n\in \mathbb{N}: (n,q)=1, (n+r,q)=1\}$ for a given integer $r$. In particular, it reduces to the Estermann $D$-function $D(s,\lambda,\frac h{l})$ if $S_q=\mathbb{N}$.

\subsection{The generalized Estermann $D$-function with a character}
Suppose that $q, l, h$ are integers, and $\chi$ is a character modulo $q$.  For any given $\lambda\in\mathbb{C}$, the generalized Estermann $D$-function with a character is defined as
\begin{align}
D\left(s,\lambda,\frac h{lq}, \chi\r)=\sum_{n}\frac{\sigma_\lambda(n)\chi(n)}{n^s}e\left(n\frac h{lq}\r).
\end{align}
\begin{lemma}\label{lemGED}
For any fixed $\lambda\in\mathbb{C}$, $D(s,\lambda,\frac h{lq}, \chi)$ is meromorphic as a function of $s$, satisfying the functional equation
\begin{align}\label{eqfeqD}
D\left(s,\lambda,\frac h{lq},\chi\r)&=2(2\pi)^{-2-\lambda+2s}(lq)^{\lambda-2s}\Gamma\left(1-s\r) \Gamma\left(1+\lambda-s\r)\\
&\times\left[-\cos\left(\frac{\pi}2\left(2s-\lambda\r)\r) \mathcal{A}_1\left(1-s,\lambda,\frac h{lq},\chi\r)+\cos\left(\frac{\pi\lambda}2\r)\mathcal{A}_2\left(1-s,\lambda,\frac h{lq},\chi\r)\r],\notag
\end{align}
where
\begin{align}\label{eqA1-2}
\mathcal{A}_1\left(s,\lambda,\frac h{lq},\chi\r)&=\sum_{1\le u,v\le lq}\chi(uv)e\left(\frac{uvh}{lq}\r)F\left(s,\frac{u}{lq}\r) F\left(s+\lambda,\frac{v}{lq}\r),\\
\mathcal{A}_2\left(s,\lambda,\frac h{lq},\chi\r)&=\sum_{1\le u,v\le lq}\chi(uv)e\left(\frac{uvh}{lq}\r)F\left(s,\frac{u}{lq}\r) F\left(s+\lambda,-\frac{v}{lq}\r).
\end{align}
If $\lambda\neq 0$, then $D$ has two simple poles at $s=1$ and $s=1+\lambda$ with respective residues
\begin{align}
&(lq)^{-2+\lambda}\sum_{1\le u,v\le lq}\chi(uv)e\left(\frac{uvh}{lq}\r) \zeta\left(1-\lambda,\frac{v}{lq}\r), \notag\\
&(lq)^{-2-\lambda}\sum_{1\le u,v\le lq}\chi(uv)e\left(\frac{uvh}{lq}\r)\zeta\left(1+\lambda,\frac{u}{lq}\r).\notag
\end{align}
\end{lemma}
\begin{proof}
The functional equation has already been obtained by the author in \cite[Lemma 4.3]{Wu19}, whose focus is on primitive character. However, the primitive character condition has not been used in the proof of the functional equation, and so it also applies here.
The two residues follow easily by calculating in the expression
\begin{align}\label{eqD}
D\left(s,\lambda,\frac h{lq},\chi\r)=(lq)^{-2s+\lambda}\sum_{1\le u,v\le lq}\chi(uv)e\left(\frac{uvh}{lq}\r)\zeta\left(s,\frac{u}{lq}\r) \zeta\left(s-\lambda,\frac{v}{lq}\r)
\end{align}
the residue of the Hurwitz zeta-function at the pole $1$.
\end{proof}

Note in the functional equation $\eqref{eqfeqD}$ that the coprime condition $(h,lq)=1$ is not assumed. For the special case with $(h,lq)=1$ and $\chi=\chi_0$, we may simplify $\mathcal{A}_i$ by the following lemma.

\begin{lemma}\label{lemiD}
Let $h,l,q$ be three integers satisfying $(h,lq)=1$ and $\chi_0$ be the principal character modulo $q$. For any given $\lambda\in\mathbb{C}$, we have
\begin{align}
D\left(s,-\lambda,-\frac {\ol{h}}{lq}\r)=\frac1{lq}\sum_{1\le uv\le lq}e\left(\frac{uvh}{lq}\r)\sum_{m}\frac1{m^s}e\left(\frac{mu}{lq}\r) \sum_{n}\frac1{n^{s+\lambda}}e\left(\frac{nv}{lq}\r)\notag
\end{align}
and
\begin{align}
D\left(s,-\lambda,-\frac {\ol{h}}{lq},\chi_0\r)=\frac1{lq}\sum_{1\le uv\le lq}e\left(\frac{uvh}{lq}\r)\sum_{(m,q)=1}\frac1{m^s}e\left(\frac{mu}{lq}\r) \sum_{(n,q)=1}\frac1{n^{s+\lambda}}e\left(\frac{nv}{lq}\r).\notag
\end{align}
\end{lemma}
\begin{proof}
We show the second identity in full details only, because the treatment of the first identity is similar and easier.
For the sums on the right-hand side of the second identity, we sum over $v$ first to see that the sum vanishes unless $lq\mid n+uh$, in which case it equals $lq$.
Thus, after an exchange of the summation, the right-hand side evolves into
\begin{align}
\sum_{1\le u\le lq}\sum_{(m,q)=1}\frac1{m^s}e\left(\frac{mu}{lq}\r) \sum_{\substack{n\equiv -uh (\bmod lq)\\ (n,q)=1}} \frac1{n^{s+\lambda}} =\sum_{(mn,q)=1}\frac{1}{m^sn^{s+\lambda}}e\Bigg(-\frac{mn\ol{h}}{lq}\Bigg),\notag
\end{align}
 and this establishes the lemma.
\end{proof}

\subsection{The $\mathcal{D}_{q}$-function}\label{sec5.2}
For integers $q, l, r, h$ and any $\lambda\in\mathbb{C}$, we define the $\mathcal{D}_{q}$-function as
\begin{align}
\mathcal{D}_{q}\left(s,\lambda,\frac h{l},r\r)=\sum_{\substack{(n,q)=1\\ (n+r,q)=1}}\frac{\sigma_\lambda(n)} {n^s} e\left(n\frac h{l}\r).\notag
\end{align}
Obviously, there is $\mathcal{D}_{q}=\mathcal{D}_{q^*}$ with $q^*$ being the largest square-free divisor of $q$.

\begin{proposition}\label{lemgED2}
Let $q, l, r, h$ be integers, satisfying $(l,qh)=1$. For any fixed $\lambda\in\mathbb{C}$, $\mathcal{D}_{q}(s,\lambda,\frac h{l},r)$ is meromorphic as a function of $s$, satisfying the functional equation
\begin{align}
\mathcal{D}_{q}&\left(\tfrac12+s,\lambda,\frac h{l},r\r)=2(2\pi)^{-1-\lambda+2s}\Gamma\left(\tfrac12-s\r) \Gamma\left(\tfrac12+\lambda-s\r)\notag\\
&\ \ \ \ \ \ \ \ \ \ \times\sum_{a\mid q}\sum_{b\mid q_r}\frac{\mu(ab)}{ab}\sum_{a_1\mid a}\sum_{b_1\mid b} (la_1b_1)^{\lambda-2s} \msum_{i(\bmod a_1)}\msum_{j(\bmod b_1)}
e\left(\frac{jr}{b_1}\r) \notag\\
&\ \ \ \ \ \ \ \ \ \ \times\left[\sin\left(\pi\left(s-\frac\lambda2\r)\r) D\Bigg(\tfrac12-s,-\lambda,-\frac{\ol{h_{i,j}}}{la_1b_1},\chi'_0\Bigg) +\cos\left(\frac{\pi\lambda}2\r) D\Bigg(\tfrac12-s,-\lambda,\frac{\ol{h_{i,j}}}{la_1b_1},\chi'_0\Bigg)\r]\notag
\end{align}
with $\chi'_0$ being the principal character modulo $a_1b_1$ and
\begin{align}\label{eqhij}
\left\{
\begin{aligned}
&h_{i,j}\equiv ilb_1  \pmod{a_1},\\
&h_{i,j}\equiv jla_1 \pmod{b_1},\\
&h_{i,j}\equiv ha_1b_1 \pmod{l}.
\end{aligned}
\right.
\end{align}
If $\lambda\neq 0$, then $\mathcal{D}_{q}$ has simple poles at $s=1$ and $s=1+\lambda$ with respective residues
\begin{align}
\frac{\vp(q)}{q}\sum_{b\mid q_r}\frac{\mu(b)}{\vp(b)} l^{-1+\lambda}\zeta_q(1-\lambda),\ \ \ \ \ \ \ \ \ \ \ \ \ \ \ \ \  \frac{\vp(q)}{q}\sum_{b\mid q_r}\frac{\mu(b)}{\vp(b)} l^{-1-\lambda}\zeta_q(1+\lambda).\notag
\end{align}
\end{proposition}

\begin{proof}
Note that a replacement of $q$ with $q^*$ does not affect the functional equation, as well as the two residues. Thus, it is sufficient to prove the proposition for square-free $q$, a convention that holds through the proof.
By expressing the coprime conditions in terms of multiplicative characters with \eqref{eqcoprimechi}, we have
\begin{align}\label{eqDsharpD}
\mathcal{D}_{q}\left(s,\lambda,\frac h{l},r\r)=\sum_{b\mid q}\frac{\mu(b)}{\vp(b)} \sum_{\chi(\bmod b)}\ol{\chi}(-r)D\left(s,\lambda,\frac {hq}{lq},\chi_0\chi\r),
\end{align}
where $\chi_0$ is the principal character modulo $q$.
We apply the functional equation \eqref{eqfeqD} to the generalized Estermann $D$-function. It follows that
\begin{align}\label{eqD_q+1}
\mathcal{D}_{q}\left(s,\lambda,\frac h{l},r\r)=&2(2\pi)^{-2-\lambda+2s}(lq)^{\lambda-2s} \Gamma(1-s)\Gamma(1+\lambda-s)\\
&\times\left[-\cos\left(\frac{\pi}2(2s-\lambda)J_1(1-s) +\cos\left(\frac{\pi\lambda}{2}\r)J_2(1-s)\r)\r]\notag
\end{align}
with
\begin{align}
J_i(s)=\sum_{b\mid q}\frac{\mu(b)}{\vp(b)}\sum_{\chi (\bmod b)}\ol{\chi}(-r)\mathcal{A}_i\left(s,\lambda,\frac {hq}{lq},\chi_0\chi\r)\notag
\end{align}
for $i=1,2$.
Applying the expression \eqref{eqA1-2} and then executing the sum over $b$ and $\chi$ with \eqref{eqcoprimechi} again, we have
\begin{align}\label{eqchiJ1}
J_1(s)=\sum_{\substack{1\le u,v\le lq\\ (uv,q)=1\\ (uv+r,q)=1}}e\left(\frac{uvh}{l}\r) F\left(s,\frac{u}{lq}\r) F\left(s+\lambda,\frac{v}{lq}\r).
\end{align}

To evaluate the sum in \eqref{eqchiJ1}, we should remove the coprime conditions on $uv$ again.
With $q_r$ being the largest factor of $q$ coprime with $r$ and $q=(q,r)q_{_r}$, it follows that
\begin{align}
\begin{array}{c}
(uv,q)=1\\
(uv+r,q)=1
\end{array}\Longleftrightarrow
\begin{array}{c}
(uv,q)=1\\
(uv+r,q_r)=1
\end{array}.\notag
\end{align}
Applying this and writing the coprime conditions in terms of M\"{o}bius function, we have
\begin{align}
\sum_{\substack{1\le u,v\le lq\\ (uv,q)=1\\ (uv+r,q)=1}}=\sum_{a\mid q}\mu(a)\sum_{b\mid q_r}\mu(b)\sum_{\substack{1\le u,v\le lq\\ a\mid uv\\ b\mid uv+r}}=\sum_{a\mid q}\sum_{b\mid q_r}\mu(ab)\sum_{\substack{1\le u,v\le lq\\ a\mid uv\\ b\mid uv+r}}\notag
\end{align}
by assuming $(a,b)=1$ at no cost.

Besides, we apply \eqref{eqdivisor} to write the conditions $a\mid uv$ and $b\mid uv+r$ in terms of the Ramanujan sum, and then it follows that
\begin{align}
\sum_{\substack{1\le u,v\le lq\\ (uv,q)=1\\ (uv+r,q)=1}}=\sum_{a\mid q}\sum_{b\mid q_r}\frac{\mu(ab)}{ab}\sum_{a_1\mid a}\sum_{b_1\mid b}\sum_{1\le u,v\le lq}c_{a_1}(uv)c_{b_1}(uv+r).\notag
\end{align}
Inserting this into \eqref{eqchiJ1} and applying the exponential sum formula to the Ramanujan sum,
then after a simple arrangement, we have
\begin{align}\label{eqJ1lq}
J_1(s)=&\sum_{a\mid q}\sum_{b\mid q_r}\frac{\mu(ab)}{ab}\sum_{a_1\mid a}\sum_{b_1\mid b} \mathop{\sum\nolimits^*}_{i(\bmod a_1)}\mathop{\sum\nolimits^*}_{j(\bmod b_1)}
e\left(\frac{jr}{b_1}\r)\\
&\times\sum_{1\le u,v\le lq} e\left(\frac{uv(ha_1b_1+jla_1+ilb_1)}{la_1b_1}\r) \sum_{m}\frac1{m^{s}}e\left(\frac{mu}{lq}\r) \sum_{n}\frac1{n^{s+\lambda}}e\left(\frac{nv}{lq}\r).\notag
\end{align}

In the expression \eqref{eqJ1lq}, we should note that there are many terms in the sum over $u,v$, which make no essential contribution to $J_1$. In particular, all the terms with $(u,q)>1$ or $(v,q)>1$ are in the case, which will certainly vanish if we sum over $a,b,a_1,b_1,i,j$ first. What is more, we note that many terms in the summation over $m,n$ will vanish if we sum over other variables first. As we will see, these nonessential terms do not contribute to $\mathcal{A}_1$, but lead to obstructions when we bound the error with the functional equation. Thus, we should kick them out now. More precisely, we observe that all the terms in the summation over $m,n$ satisfying at least one of the following three conditions
$(mn,a_1b_1)>1$, $q_{a_1b_1}\nmid m$, and $q_{a_1b_1}\nmid n$
can be removed at no cost. Let us see this in detail.
For notational convenience, we apply the following notation
\begin{align}
h_{i,j}=ha_1b_1+jla_1+ilb_1.\notag
\end{align}
Then it is easy to see that $(h_{i,j},la_1b_1)=1$
and
\begin{align}
\left\{
\begin{aligned}
&h_{i,j}\equiv ilb_1  \pmod{a_1},\\
&h_{i,j}\equiv jla_1 \pmod{b_1},\\
&h_{i,j}\equiv ha_1b_1 \pmod{l}.
\end{aligned}
\right.\notag
\end{align}
We show in full detail for the summation over $m$, and the case for $n$ is identical.
More precisely, we first kick out all nonessential terms with $(v,q)>1$ at no cost, and then the $u$-sum in \eqref{eqJ1lq} is
\begin{align}
\sum_{1\le u\le lq}e\left(\frac{vh_{i,j}q_{a_1b_1}+m}{lq}u\r)=\left\{
\begin{aligned}
&lq \ \ \ \text{for} \ \ \ vh_{i,j}q_{a_1b_1}\equiv -m \pmod{lq},\\
&0\ \ \ \ \text{otherwise}.\notag
\end{aligned}
\right.
\end{align}
Here $q_{a_1b_1}=q/a_1b_1$ since $q$ is square free.
This indicates that contribution of the $m$-sum arises just from such terms that
$-m\equiv vh_{i,j}q_{a_1b_1}~ (\bmod lq)$,
in particular $q_{a_1b_1}\mid m$ and $(m,a_1b_1)=1$ since $(vh_{i,j}q_{a_1b_1},a_1b_1)=1$. An identical discussion shows the same result for the summation over $n$.
Thus, we add
\[
(m,a_1b_1)=1,\ \ \ q_{a_1b_1}\mid m, \ \ \ (n,a_1b_1)=1,\ \ \ q_{a_1b_1}\mid n
\]
to the sum at no cost.

We extract $q_{a_1b_2}$ from the sum to get
\begin{align}
J_1(s)=&\sum_{a\mid q}\sum_{b\mid q_r}\frac{\mu(ab)}{ab}\sum_{a_1\mid a}\sum_{b_1\mid b} (q_{a_1b_1})^{2-2s-\lambda} \mathop{\sum\nolimits^*}_{i(\bmod a_1)}\mathop{\sum\nolimits^*}_{j(\bmod b_1)}
e\left(\frac{jr}{b_1}\r)\notag\\
&\times\sum_{1\le u,v\le la_1a_2}e\left(\frac{uvh_{i,j}}{la_1b_1}\r)\sum_{(m,a_1b_1)=1}\frac1{m^{s}} e\left(\frac{mu}{la_1b_1}\r) \sum_{(n,a_1b_1)=1}\frac1{n^{s+\lambda}}e\left(\frac{nv}{la_1b_1}\r).\notag
\end{align}
Then by the second identity of Lemma \ref{lemiD},
\begin{align}
J_1(s)=&\sum_{a\mid q}\sum_{b\mid q_r}\frac{\mu(ab)}{ab}\notag\\
&\times\sum_{a_1\mid a}\sum_{b_1\mid b} (q_{a_1b_1})^{2-2s-\lambda} la_1b_1\mathop{\sum\nolimits^*}_{i(\bmod a_1)}\mathop{\sum\nolimits^*}_{j(\bmod b_1)}
e\left(\frac{jr}{b_1}\r)D\Bigg(s,-\lambda,-\frac{\ol{h_{i,j}}}{la_1b_1},\chi'_0\Bigg),\notag
\end{align}
with $\chi'_0$ being the principal character modulo $a_1b_1$.
A similar deduction shows
\begin{align}
J_2(s)=&\sum_{a\mid q}\sum_{b\mid q_r}\frac{\mu(ab)}{ab}\sum_{a_1\mid a}\sum_{b_1\mid b} (q_{a_1b_1})^{2-2s-\lambda} la_1b_1 \mathop{\sum\nolimits^*}_{i(\bmod a_1)}\mathop{\sum\nolimits^*}_{j(\bmod b_1)}
e\left(\frac{jr}{b_1}\r)D\Bigg(s,-\lambda,\frac{\ol{h_{i,j}}}{la_1b_1},\chi'_0\Bigg).\notag
\end{align}
Applying these two expressions into \eqref{eqD_q+1}, one obtains the functional equation after making the change of variables $s\rightarrow \frac12+s$.

To calculate the residues, we apply Lemma \ref{lemGED} into \eqref{eqDsharpD} to find that the residue at $1$ is
\begin{align}
(lq)^{-2+\lambda}\sum_{b\mid q}\frac{\mu(b)}{\vp(b)}\sum_{\chi (\bmod b)}\ol{\chi}(-r)\sum_{1\le u,v\le lq}\chi_0(uv)\chi(uv)e\left(\frac{uvh}{l}\r) \zeta\left(1-\lambda,\frac{v}{lq}\r).\notag
\end{align}
An application of \eqref{eqcoprimechi} as before gives
\begin{align}
res(1)=(lq)^{-2+\lambda}\sum_{\substack{1\le u,v\le lq\\ (uv,q)=1\\ (uv+r,q)=1}} e\left(\frac{uvh}{l}\r) \zeta\left(1-\lambda,\frac{v}{lq}\r).\notag
\end{align}
Now we apply
\begin{align}
\begin{array}{c}
(uv,q)=1\\
(uv+r,q)=1
\end{array}\Longleftrightarrow
\begin{array}{c}
(uv,q)=1\\
(uv+r,q_r)=1\notag
\end{array}
\end{align}
and then write the coprime condition $(uv+r,q_r)=1$ in terms of the Ramanujan sum with \eqref{eqcoprime} to find
\begin{align}\label{+eq2.67}
res(1)=&(lq)^{-2+\lambda}\sum_{b\mid q_r}\frac{\mu(b)}{b}\sum_{b_1\mid b} \sum_{\substack{1\le u,v\le lq\\ (uv,q)=1}}c_{b_1}(uv+r)e\left(\frac{uvh}{l}\r) \zeta\left(1-\lambda,\frac{v}{q}\r)\\
=&(lq)^{-2+\lambda}\sum_{b\mid q_r}\frac{\mu(b)}{b}\sum_{b_1\mid b}\mathop{\sum\nolimits^*}_{j(\bmod b_1)} e\left(\frac{jr}{b_1}\r) \sum_{\substack{1\le u,v\le lq\\ (uv,q)=1}}e\left(\frac{uv(hb_1+jl)}{lb_1}\r)  \zeta\left(1-\lambda,\frac{v}{lq}\r).\notag
\end{align}
With $u=iq+k$, the sum over $u$ is equal to
\begin{align}
\mathop{\sum\nolimits^*}_{k(\bmod q)} e\left(\frac{vk(hb_1+jl)}{lb_1}\r) \sum_{i(\bmod l)}e\left(\frac{ivq_{b_1}(hb_1+jl)}{l}\r).\notag
\end{align}
Since $(q_{b_1}(hb_1+jl),l)=1$, the last sum over $j$ vanishes unless $l\mid v$, in which case it is equal to $l$.
Applying this into \eqref{+eq2.67}, we make the variable change $v\rightarrow lv$ and then $vk\rightarrow k$ in the sum, and it follows that
\begin{align}
res(1)=&l(lq)^{-2+\lambda}\sum_{b\mid q_r}\frac{\mu(b)}{b}\sum_{b_1\mid b} \mathop{\sum\nolimits^*}_{v(\bmod q)} \zeta\left(1-\lambda,\frac{v}{q}\r)\mathop{\sum\nolimits^*}_{j(\bmod b_1)} e\left(\frac{jr}{b_1}\r)\mathop{\sum\nolimits^*}_{k(\bmod q)}e\left(\frac{k(hb_1+jl)}{b_1}\r). \notag
\end{align}
Since $(hb_1+jl,b_1)=1$, we apply the Chinese Remainder Theorem to reduce the exponential sum over $k$ to a Ramanujan sum that
\begin{align}
\mathop{\sum\nolimits^*}_{k(\bmod q)}e\left(\frac{k(hb_1+jl)}{b_1}\r) &=\vp(q_{b_1})c_{b_1}(hb_1+jl)=\vp(q)\frac{\mu(b_1)}{\vp(b_1)}.\notag
\end{align}
Subsequently, the sum over $j$ also turns out to be a Ramanujan sum, which is $c_{b_1}(r)=\mu(b_1)$ as $(b_1,r)=1$.
Thus, we conclude that
\begin{align}
res(1)=&l(lq)^{-2+\lambda} \mathop{\sum\nolimits^*}_{v(\bmod q)} \zeta\left(1-\lambda,\frac{v}{q}\r)\Bigg(\vp(q)\sum_{b\mid q_r}\frac{\mu(b)}{b}\sum_{b_1\mid b}\frac{\mu^2(b_1)}{\vp(b_1)}\Bigg)\notag\\
=&\frac{\vp(q)}{q}\sum_{b\mid q_r}\frac{\mu(b)}{\vp(b)} l^{-1+\lambda}\zeta_q(1-\lambda),\notag
\end{align}
observing that
\[
\vp(q)\sum_{b\mid q_r}\frac{\mu(b)}{b}\sum_{b_1\mid b}\frac{\mu^2(b_1)}{\vp(b_1)}=\vp(q)\sum_{b\mid q_r}\frac{\mu(b)}{\vp(b)}.
\]
The treatment for the residue at $1+\lambda$ is identical.
\end{proof}

\section{Initial treatment for the off-diagonal terms}\label{secnext}

\subsection{Some arithmetic sums}
We require the computation of some arithmetic sums.

\begin{lemma}\label{lemus}
For any integers $q,n$ with $(n,q)=1$, $\re(\alpha)<0$ we have
\begin{align}
\sigma_{\alpha}(n)=\zeta_{q}(1-\alpha)\sum_{(l,q)=1} \frac{c_{l}(n)}{l^{1-\alpha}}.\notag
\end{align}
\end{lemma}
\begin{proof}
This is an analog of \cite[Lemma 5.2]{You11}, which follows immediately from the formula $c_l(n)=\sum_{d\mid(l,n)}d\mu(l/d)$ with a reversal of summations.
\end{proof}

\begin{lemma}\label{lemsumf}
Let $k\mid q$, $\re(s)>1$, and $\re(\lambda)>-1$. Then
\begin{align}
\sum_{(r,q)=k} \frac{1}{r^s}\sum_{(l,q)=1}\frac{c_{l}(r)} {l^{2+\lambda}} =\frac{1}{k^s}\prod_{\substack{p\mid k\\ p\nmid q/k}}\left(1-\frac1{p^s}\r)^{-1}\frac{\zeta_q(s)\zeta_q(1+\lambda+s)} {\zeta_q(2+\lambda)}.\notag
\end{align}
\end{lemma}
\begin{proof}
With the change $r\rightarrow k^{\sharp}r$, we split the sum over $r$ into two sums over both $k^{\sharp}$ and $r$, where $k^{\sharp}$ is over all positive integers owning the same distinct prime factors as $k$, and $r$ is over the positive integers with $(r,q)=1$. So we have
$(k^{\sharp}/k,q/k)=1$, $(k^{\sharp},l)=1$,
which mean
\begin{align}
c_{l}(k^{\sharp}r)=c_l(r),\ \ \ \ \ \ \ \ \ \ \sum_{k^{\sharp}}\frac1{\mathop{k^{\sharp}}^s}=\frac1{k^s}\prod_{\substack{p\mid k\\ p\nmid q/k}}\left(1-\frac1{p^s}\r)^{-1},\notag
\end{align}
and then
\begin{align}
\sum_{(r,q)=k} \frac{1}{r^s}\sum_{(l,q)=1}\frac{c_{l}(r)} {l^{2+\lambda}} =\frac{1}{k^s}\prod_{\substack{p\mid k\\ p\nmid q/k}}\left(1-\frac1{p^s}\r)^{-1} \sum_{(r,q)=1}\frac{1}{r^s}\sum_{(l,q)=1}\frac{c_l(r)}{l^{2+\lambda}}.\notag
\end{align}
After we execute the sum over $l$ by Lemma \ref{lemus}, it becomes
\begin{align}
\frac{1}{k^s}\prod_{\substack{p\mid k\\ p\nmid q/k}}\left(1-\frac1{p^s}\r)^{-1}\frac1{\zeta_q(2+\lambda)}\sum_{(r,q)=1} \frac{\sigma_{-1-\lambda}(r)}{r^s} =\frac{1}{k^s}\prod_{\substack{p\mid k\\ p\nmid q/k}}\left(1-\frac1{p^s}\r)^{-1}\frac{\zeta_q(s)\zeta_q(1+\lambda+s)} {\zeta_q(2+\lambda)},\notag
\end{align}
which establishes the lemma.
\end{proof}

We require an approximate functional equation for the divisor function, which will be applied to separate variables in $\sigma_\lambda(n+r)$.
\begin{lemma}\label{lemafed}
Let $n$ be any positive integer.
For $\lambda\in \mathbb{C}$, we have
\begin{align}\label{eqafed}
\sigma_\lambda(n)=&\sum_{(l,q)=1}\frac{c_{l}(n)}{l^{1-\lambda}} \varpi _\lambda\left(\frac l{\sqrt{n}}\r)+n^\lambda \sum_{(l,q)=1}\frac{c_{l}(n)}{l^{1+\lambda}}\varpi _{-\lambda}\left(\frac l{\sqrt{n}}\r),
\end{align}
where
\begin{align}
\varpi_\lambda(x)=\int_{(a)}x^{-w}\zeta_q(1-\lambda+w)\frac{G(w)}wdw\notag
\end{align}
with $a>|\re(\lambda)|$ and $G$ defined in Definition \ref{DefG}.
\end{lemma}

\begin{proof}
The proof is identical to \cite[Lemma 5.4]{You11} but with \cite[Lemma 5.2]{You11} being replaced by Lemma \ref{lemus}.
\end{proof}

The following lemma is required in the deduction of the main terms.
\begin{lemma}\label{lemsumk}
For any nonzero $z\in\mathbb{C}$, we have
\begin{align}\label{eqsumk}
&\Bigg(\prod_{p\mid q}\frac{p-2}{p-1}\Bigg)\sum_{k\mid q}\frac1 {k^z}\Bigg(\prod_{p\mid k}\frac{p-1}{p-2}\Bigg)\Bigg(\prod_{\substack{p\mid k\\ p\nmid q/k}}\left(1-\frac1{p^{z}}\r)^{-1}\Bigg)\sum_{d\mid k}\vp(d)\mu\left(\frac{q}{d}\r)\\
&=\frac{\vp^*(q)q}{\vp(q)q^z}\prod_{p\mid q}\left(1-\frac1{p^z}\r)^{-1}\left(1-\frac1{p^{1-z}}\r).\notag
\end{align}
\end{lemma}
\begin{proof}
It is an exercise to check that both sides of the identity are multiplicative functions of $q$. Thus, it is sufficient to check the identity for $q=p^m$, $m\ge1$. Note that the term $\mu\left(\frac qd\r)$ vanishes unless $\frac qd$ is equal to $1$ or a prime number, which implies that $k$ can only take values at $p^m$ and $p^{m-1}$. We check the identity in two cases according to $m\ge2$ and $m=1$. For $m\ge 2$, a direct calculation gives that the left-hand side in \eqref{eqsumk} is equal to
\begin{align}
&\frac{p-2}{p-1}\left[\frac{p-1}{p-2}\frac1 {p^{mz}}\left(1-\frac1{p^{z}}\r)^{-1}\left(\vp\left(p^m\r)-\vp\left(p^{m-1}\r)\r) -\frac{p-1}{p-2}\frac1 {p^{(m-1)z}}\vp\left(p^{m-1}\r)\r]\notag\\
=&\frac{\vp^*(p^m)p^m}{\vp(p^m)p^{mz}}\left(1-\frac1{p^z}\r)^{-1}\left(1-\frac1{p^{1-z}}\r).\notag
\end{align}
For $m=1$, the left-hand side in \eqref{eqsumk} is
\begin{align}
\frac{p-2}{p-1}\left[\frac{p-1}{p-2}\frac1 {p^{z}}\left(1-\frac1{p^{z}}\r)^{-1}\left(\vp\left(p\r)-1\r)-1\r]
=&\frac{\vp^*(p)p}{\vp(p)p^z}\left(1-\frac1{p^z}\r)^{-1}\left(1-\frac1{p^{1-z}}\r).\notag
\end{align}
Thus, we complete the proof.
\end{proof}

\subsection{Separation of variables}\label{secsv}
The separation of variables is almost the same as in \cite{You11}, and we sketch it here. By writing $n=m+r$ in the sum of $B_{M,N}$, we have
\begin{align}
B_{M,N}&=\frac1{\vp^*(q)}\sum_{d\mid q}\vp(d)\mu\left(\frac qd\r) \sum_{r\equiv0~(\bmod d)}\\
&\times\sum_{\substack{(m,q)=1\\ (m+r,q)=1}}\frac{\sigma_{\alpha-\beta}(m) \sigma_{\gamma-\delta}(m+r)} {m^{\frac12+\alpha}(m+r)^{\frac12+\gamma}}
V_{\alpha,\beta,\gamma,\delta} \left(\frac{m(m+r)}{q^2}\r) W\left(\frac mM\r)W\left(\frac{m+r}N\r). \notag
\end{align}
Similarly, with $r=m+n$,
\begin{align}
B_{\ol{M,N}}&=\frac1{\vp^*(q)}\sum_{d\mid q}\vp(d)\mu\left(\frac qd\r) \sum_{r\equiv0~(\bmod d)}\notag\\
&\times\sum_{\substack{(m,q)=1\\ (r-m,q)=1}}\frac{\sigma_{\alpha-\beta}(m) \sigma_{\gamma-\delta}(r-m)} {m^{\frac12+\alpha}(r-m)^{\frac12+\gamma}}
V_{\alpha,\beta,\gamma,\delta} \left(\frac{m(r-m)}{q^2}\r) W\left(\frac mM\r)W\left(\frac{r-m}N\r).\notag
\end{align}
These two expressions are almost the same as in \cite{You11}, but with extra coprime conditions
$(m,q)=1$, $(m+r,q)=1$, $(r-m,q)=1$ in the sum. The coprime conditions do not affect the separation of variables, except the expansion of $\sigma_{\gamma-\delta}(m+r)$ (or $\sigma_{\gamma-\delta}(r-m)$ in $B_{\ol{M,N}}$) and the expression of the sum over $m$ in terms of $\mathcal{D}_q$-function.

We apply the approximate functional equation \eqref{eqafed} to expand $\sigma_{\gamma-\delta}(m+r)$ (or $\sigma_{\gamma-\delta}(r-m)$ in $B_{\ol{M,N}}$), which splits $B_{M,N}$ (or $B_{\ol{M,N}}$) into two parts
\begin{align}\label{eqBMN}
B_{M,N}=C_{M,N}+\wt{C}_{M,N},\ \ \ \ \ \ \ \ \ B_{\ol{M,N}}=C_{\ol{M,N}}+\wt{C}_{\ol{M,N}},
\end{align}
where $C_{M,N}$ (or $C_{\ol{M,N}}$) is the contribution from the first part of the approximate functional equation and $\wt{C}_{M,N}$ (or $\wt{C}_{\ol{M,N}}$) from the second part. More precisely, we have
\begin{align}\label{eqCMN}
C_{M,N}&=\frac1{\vp^*(q)}\sum_{d\mid q}\vp(d)\mu\left(\frac qd\r) \sum_{r\equiv0~(\bmod d)}\sum_{(l,q)=1}\frac1{l^{1-\gamma+\delta}}\\
&\times\sum_{\substack{(m,q)=1\\ (m+r,q)=1}}\frac{\sigma_{\alpha-\beta}(m) c_l(m+r)} {m^{\frac12+\alpha}(m+r)^{\frac12+\gamma}}
V_{\alpha,\beta,\gamma,\delta} \left(\frac{m(m+r)}{q^2}\r) W\left(\frac mM\r)W\left(\frac{m+r}N\r)\varpi _{\gamma-\delta}\left(\frac l{\sqrt{m+r}}\r), \notag
\end{align}
and an identical expression for $C_{\ol{M,N}}$ except replacing all $m+r$ terms with $r-m$ only. It is easy to see that
\begin{align}\label{eqwtCC}
\wt{C}_{M,N}(\alpha,\beta,\gamma,\delta)=C_{M,N}(\alpha,\beta,\delta,\gamma),\ \ \ \ \ \ \ \wt{C}_{\ol{M,N}}(\alpha,\beta,\gamma,\delta)=C_{\ol{M,N}}(\alpha,\beta,\delta,\gamma).
\end{align}

For notational convenience, we apply the following notations
\begin{align}\label{eqdefH1}
H_1(s,u_1,u_2,w)=\frac{G(s)G(w)}{sw} g_{\alpha,\beta,\gamma,\delta}(s) \wt{W}(u_1)\wt{W}(u_2) \zeta_q(1-\gamma+\delta+w),
\end{align}
\begin{align}\label{eqH}
H(s,u_1,u_2,v,w)=\frac{\Gamma(v) \Gamma(\frac12+\gamma+s+u_2-v-\frac w2)}{\Gamma(\frac12+\gamma+s+u_2-\frac w2)}H_1(s,u_1,u_2,w),
\end{align}
\begin{align}
\ol{H}(s,u_1,u_2,v,w)=\frac{\Gamma(v)\Gamma(\frac12-\gamma-s-u_2+\frac w2)}{\Gamma(\frac12-\gamma-s-u_2+v+\frac w2)}H_1(s,u_1,u_2,w).
\end{align}
By Stirling's approximation, it is easy to see that $H_1$ and $H$ decay rapidly as any one of the variables gets large in the imaginary direction. Also, $\ol{H}$ decays rapidly on almost all variables except $v$, which has at most a polynomial decay (or growth) on $\im(v)$. One may refer to \cite{You11} for more detail.

Now, separating the variables the same as \cite[Lemma 5.5]{You11}, we have the following two expressions.
\begin{lemma}\label{lemC_{M,N}}
With $c_s=c_v=\frac14+\ve$, $c_{u_1}=c_{u_2}=0$, and $c_w=\ve$, we have
\begin{align}\label{eqC_MN}
C_{M,N}&=\frac1{\vp^*(q)}\sum_{d\mid q}\vp(d)\mu\left(\frac qd\r)
\sum_{{r\equiv0~(\bmod d)}}\frac{1}{r^{\frac12+\gamma}}\sum_{(l,q)=1}\frac1{l^{1-\gamma+\delta}}\msum_{h(\bmod l)}e\left(\frac{hr}{l}\r) \left(\frac1{2\pi i}\r)^5\\
&\times\int_{(c_s)}\int_{(c_w)}\int_{(c_{u_1})}\int_{(c_{u_2})} \int_{(c_v)}\frac{q^{2s}M^{u_1}N^{u_2}} {r^{s+u_2-v-\frac w2}l^w}\mathcal{D}_{q}\left(\tfrac12+\alpha+s+u_1+v,\alpha-\beta,\frac h{l},r\r)\notag\\
&\ \ \ \ \ \ \ \ \ \ \ \ \ \ \ \ \ \ \ \ \ \ \ \ \ \ \ \ \ \ \ \ \ \ \ \ \ \ \ \ \ \ \ \ \ \ \ \ \ \ \ \ \ \ \ \ \ \ \ \ \ \ \times H(s,u_1,u_2,v,w)dvdu_2du_1dwds. \notag
\end{align}
\end{lemma}
\begin{lemma}\label{lemC_ol{M,N}}
With $c_{s}=c_v=\frac14+\ve$, $c_{u_1}=c_{u_2}=0$, and $c_w=2$, we have
\begin{align}\label{eqC_{M,N}}
C_{\ol{M,N}}&=\frac1{\vp^*(q)}\sum_{d\mid q}\vp(d)\mu\left(\frac qd\r)
\sum_{{r\equiv0~(\bmod d)}}\frac{1}{r^{\frac12+\gamma}}\sum_{(l,q)=1}\frac1{l^{1-\gamma+\delta}}\msum_{h(\bmod l)}e\left(\frac{-hr}{l}\r) \left(\frac1{2\pi i}\r)^5\\
&\times\int_{(c_s)}\int_{(c_w)}\int_{(c_{u_1})}\int_{(c_{u_2})} \int_{(c_v)}\frac{q^{2s}M^{u_1}N^{u_2}} {r^{s+u_2-v-\frac w2}l^w}\mathcal{D}_{q}\left(\tfrac12+\alpha+s+u_1+v,\alpha-\beta,\frac h{l},-r\r)\notag\\
&\ \ \ \ \ \ \ \ \ \ \ \ \ \ \ \ \ \ \ \ \ \ \ \ \ \ \ \ \ \ \ \ \ \ \ \ \ \ \ \ \ \ \ \ \ \ \ \ \ \ \ \ \ \ \ \ \ \ \ \ \ \ \times \ol{H}(s,u_1,u_2,v,w)dvdu_2du_1dwds. \notag
\end{align}
\end{lemma}

By Proposition \ref{lemgED2}, we move $c_v$ to $\ve$ in the above two expressions, both crossing the two poles of $\mathcal{D}_{q}$ at $\tfrac12+s+u_1+v=1-\alpha$ and $\tfrac12+s+u_1+v=1-\beta$, whose residues contribute the main terms. The integral along the new paths contributes an error, denoted by $E_{M,N}$ (or $E_{\ol{M,N}}$), where we move $c_s$ to $\ve$ without encountering any poles. Let $P_{M,N}(\alpha,\beta,\gamma,\delta)$ (or $P_{\ol{M,N}}(\alpha,\beta,\gamma,\delta)$) denote the contribution of the former pole, and $P_{M,N}(\beta,\alpha,\gamma,\delta)$ (or $P_{\ol{M,N}}(\beta,\alpha,\gamma,\delta)$) correspond to the latter pole by symmetry. We conclude that
\begin{align}
C_{M,N}=P_{M,N}(\alpha,\beta,\gamma,\delta) +P_{M,N}(\beta,\alpha,\gamma,\delta)+E_{M,N},
\end{align}
\begin{align}\label{eqCPPE}
C_{\ol{M,N}}=P_{\ol{M,N}}(\alpha,\beta,\gamma,\delta) +P_{\ol{M,N}}(\beta,\alpha,\gamma,\delta)+E_{\ol{M,N}},
\end{align}
and we will calculate the main terms in Section \ref{secmainAO} and the error terms in Section \ref{secError}.

\section{The main term of $A_O$, $A^*_O$, and $A_{\ol{O}}$}\label{secmainAO}
We show in detail the treatment of $A_{\ol{O}}$, and cases of $A_O$ and $A^*_O$ are similar and easier.
At first, we deduce an exact expression for $P_{\ol{M,N}}(\abgd)$ in the following lemma.
\begin{lemma}\label{lemolPMN}
For $M\le N$, $c_s=\frac14$, $c_{u_1}=c_{u_2}=0$, and $c_w=\ve$, we have
\begin{align}\label{eqolPMN}
P_{\ol{M,N}}=&\zeta_q(1-\alpha+\beta)\left(\frac1{2\pi i}\r)^4\int_{(c_s)}\int_{(c_w)}\int_{(c_{u_1})}\int_{(c_{u_2})} q^{-\alpha-\gamma-u_1-u_2+\frac w2}\prod_{p\mid q}\left(1-\frac1{p^{1-(\alpha+\gamma+2s+u_1+u_2-\frac w2)}}\r)\\
&\times M^{u_1}N^{u_2}\frac{\zeta(\alpha+\gamma+2s+u_1+u_2-\tfrac w2)\zeta_q(1+\beta+\delta+2s+u_1+u_2+\tfrac w2)} {\zeta_q(2-\alpha+\beta-\gamma+\delta+w)}\notag\\
&\times \ol{H}_2(s,u_1,u_2,w) du_2du_1dwds,\notag
\end{align}
where
\begin{align}\label{eqolH_2}
\ol{H}_2(s,u_1,u_2,w)=\frac{\Gamma(\tfrac12-\alpha-u_1-s)\Gamma(\frac12-\gamma-s-u_2+\frac w2)}{\Gamma(1-\alpha-\gamma-2s-u_1-u_2+\frac w2)}H_1(s,u_1,u_2,w).
\end{align}
\end{lemma}
\begin{proof}
By Proposition \ref{lemgED2}, the residue of $\mathcal{D}_q$-function at $\tfrac12+s+u_1+v=1-\alpha$ is
\begin{align}
\frac{\vp(q)}{q}\sum_{b\mid q_r}\frac{\mu(b)}{\vp(b)} l^{-1+\alpha-\beta}\zeta_q(1-\alpha+\beta).\notag
\end{align}
Thus, we have
\begin{align}\label{eqmathcalP_ol{M,N}(t)}
P_{\ol{M,N}}=&\zeta_q(1-\alpha+\beta)\frac{\vp(q)}{q\vp^*(q)}\sum_{d\mid q}\vp(d)\mu\left(\frac qd\r)\sum_{b\mid q}\frac{\mu(b)}{\vp(b)}\sum_{\substack{r\equiv0 ~(\bmod d)\\ (r,b)=1}}\sum_{(l,q)=1}\frac{c_l(r)}{l^{2-\alpha+\beta-\gamma+\delta}}\\
&\times\left(\frac1{2\pi i}\r)^4\int_{(c_s)}\int_{(c_w)}\int_{(c_{u_1})}\int_{(c_{u_2})} \frac{q^{2s}M^{u_1}N^{u_2}} {r^{\alpha+\gamma+2s+u_1+u_2-\frac w2}l^w}\ol{H}_2\left(s,u_1,u_2,w\r)du_2du_1dwds,\notag
\end{align}
observing that $\ol{H}_2(s,u_1,u_2,w)=\ol{H}\left(s,u_1,u_2,\tfrac12-\alpha-u_1-s,w\r)$. Here the integral contours are as in Lemma \ref{lemC_ol{M,N}}.

Observing in \eqref{eqolH_2} the pole of the gamma factors, we initially move the contours such that
$c_s=\ve$, $c_{u_1}=c_{u_2}=0$, $c_w=3\ve$,
passing no poles. Then we move the $s$-line of integral to $c_s=\frac12+\ve$, crossing a pole of $\Gamma(\tfrac12-\alpha-u_1-s)$ at  $\frac12-\alpha-u_1-s=0$.

We consider the residue first. For convenience, we apply here two notations
\[
k=(q,r),\ \ \ \ \ \ \ z=1-\alpha+\gamma-u_1+u_2-\tfrac w2.
\]
Then the residue is
\begin{align}
&\zeta_q(1-\alpha+\beta)\frac{\vp(q)}{q\vp^*(q)}\sum_{d\mid q}\vp(d)\mu\left(\frac qd\r)\sum_{b\mid q}\frac{\mu(b)}{\vp(b)}\sum_{\substack{k\mid q_b\\ d\mid k}}\sum_{(r,q)=k}\sum_{(l,q)=1}\frac{c_l(r)}{l^{2-\alpha+\beta-\gamma+\delta}}\notag\\
&\times\left(\frac1{2\pi i}\r)^3\int_{(c_w)}\int_{(c_{u_1})}\int_{(c_{u_2})} \frac{q^{1-2\alpha-2u_1}M^{u_1}N^{u_2}} {r^zl^w}H_1\left(\tfrac12-\alpha-u_1,u_1,u_2,w\r)du_2du_1dw,\notag
\end{align}
where the gamma factors in $H_2$ disappear, and we may move $c_{u_2}$ to $2\ve$  to make both sums over $l$ and $r$ converge absolutely.
We apply Lemma \ref{lemsumf} to execute the sum over $l,r$, and the residue evolves into
\begin{align}\label{eqdualresidue}
&\zeta_q(1-\alpha+\beta)\frac{\vp(q)}{q\vp^*(q)}\sum_{d\mid q} \vp(d)\mu\left(\frac qd\r)\sum_{b\mid q} \frac{\mu(b)}{\vp(b)}\\
&\times \left(\frac1{2\pi i}\r)^3\int_{(c_w)}\int_{(c_{u_1})}\int_{(c_{u_2})}q^{1-2\alpha-2u_1}\sum_{\substack{k\mid q_b\\ d\mid k}}\frac{1}{k^{z}}\prod_{\substack{p\mid k\\ p\nmid q/k}}\left(1-\frac1{p^z}\r)^{-1}M^{u_1}N^{u_2}\notag\\
&\times\frac{\zeta_q(z)\zeta_q\left(2-2\alpha+\beta+\delta-u_1+u_2+\frac w2\r)} {\zeta_q(2-\alpha+\beta-\gamma+\delta+w)}H_1\left(\tfrac12-\alpha-u_1,u_1,u_2,w\r)du_2du_1dw.\notag
\end{align}

We now come to the sum over $b$, $d$, and $k$ in the above formula, whose essence is to calculate the following arithmetic sum
\begin{align}
\sum_{\substack{d\mid q\\ b\mid q}}\mu(b)\mu\left(\frac qd\r)\frac{\vp(d)}{\vp(b)}\sum_{\substack{k\mid q_b\\ d\mid k}}\frac{1}{k^z}\prod_{\substack{p\mid k\\ p\nmid q/k}}\left(1-\frac1{p^z}\r)^{-1}&=\sum_{k\mid q}\frac{1}{k^z}\prod_{\substack{p\mid k\\ p\nmid q/k}}\left(1-\frac1{p^z}\r)^{-1}\sum_{d\mid k}\vp(d)\mu\left(\frac qd\r)\sum_{b\mid q_k}\frac{\mu(b)}{\vp(b)}.\notag
\end{align}
It is easy to see that
\[
\sum_{b\mid q_k}\frac{\mu(b)}{\vp(b)}=\Bigg(\prod_{p\mid q}\frac{p-2}{p-1}\Bigg)\Bigg(\prod_{p\mid k}\frac{p-1}{p-2}\Bigg).
\]
After inserting this formula, one notes that the arithmetic sum is ready to apply Lemma \ref{lemsumk}, which gives
\[
\sum_{\substack{d\mid q\\ b\mid q}}\mu(b)\mu\left(\frac qd\r)\frac{\vp(d)}{\vp(b)}\sum_{\substack{k\mid q_b\\ d\mid k}}\frac{1}{k^z}\prod_{\substack{p\mid k\\ p\nmid q/k}}\left(1-\frac1{p^z}\r)^{-1}=\frac{\vp^*(q)q}{\vp(q)q^z}\prod_{p\mid q}\left(1-\frac1{p^z}\r)^{-1}\left(1-\frac1{p^{1-z}}\r).
\]
With this in \eqref{eqdualresidue}, the residue evolves into
\begin{align}\label{eqdualresidue1}
\zeta_q(1&-\alpha+\beta)\left(\frac1{2\pi i}\r)^3\int_{(c_w)}\int_{(c_{u_1})}\int_{(c_{u_2})} q^{-\alpha-\gamma-u_1-u_2+\frac w2} \prod_{p\mid q}\left(1-\frac1{p^{\alpha-\gamma+u_1-u_2+\frac w2}}\r)M^{u_1}N^{u_2}\\
&\times \frac{\zeta\left(1-\alpha+\gamma-u_1+u_2-\frac w2\r)\zeta_q\left(2-2\alpha+\beta+\delta-u_1+u_2+\frac w2\r)}{\zeta_q(2-\alpha+\beta-\gamma+\delta+w)}\notag\\
&\times H_1\left(\tfrac12-\alpha-u_1,u_1,u_2,w\r)du_2du_1dw.\notag
\end{align}

Now we come to the integral along the new lines, which is the right-hand side of \eqref{eqmathcalP_ol{M,N}(t)} with the contours of integration being $c_s=\frac12+\ve$, $c_{u_1}=c_{u_2}=0$, $c_w=3\ve$. Here the sum over $l,r$ converges absolutely again, and we may execute it with Lemma \ref{lemsumf} and then the sum over $b,d,k$ with Lemma \ref{lemsumk}. The calculation is just the same as we have done for the residue above. After doing this in \eqref{eqmathcalP_ol{M,N}(t)}, one notes that the integral along the new lines has the same expression as the right-hand side of \eqref{eqolPMN} except that the $s$-integral is on $c_s=\frac12+\ve$. To obtain \eqref{eqolPMN}, we move the $s$-integral back to $c_s=\frac14$. Since the pole of $\zeta(\alpha+\gamma+2s+u_1+u_2-\tfrac w2)$ is canceled by $1-\frac1{p^{1-(\alpha+\gamma+2s+u_1+u_2-\frac w2)}}$, only the pole $\frac12-\alpha-u_1-s=0$ is crossed again. It is easy to see that the new residue is also equal to \eqref{eqdualresidue1}, so the two residues cancel, and we establish the lemma.
\end{proof}

\begin{remark}
Note that Lemma \ref{lemolPMN} provides an exact expression for $P_{\ol{M,N}}$, and our proof has not encountered the obstruction as in the prime moduli case. This is due to the product $\prod_{p\mid q}\left(1-\frac1{p^{1-(\alpha+\gamma+2s+u_1+u_2-\frac w2)}}\r)$, which cancels the pole of the Riemann zeta-function and allows our moving $c_s$ back to $\frac14$. Actually, the product is a result of the coprime condition $(mn,q)=1$ in the original expression of $B_{\ol{M,N}}$. In some way, this means that the coprime condition is native and makes the true feature of the main terms.
\end{remark}

The upper bound of the main term mentioned in \eqref{eqbmain} is presented specially in the following lemma.
\begin{lemma}
For $M\le N$,
we have
\begin{align}\label{eqboundPolMN}
P_{\ol{M,N}}\ll M^{-\frac12}N^{\frac12}q^\ve.
\end{align}
\end{lemma}
\begin{proof}
Observing the location of the poles of $H_2$, we may move the contours in \eqref{eqolPMN} to
\[
c_s=\ve,\ \ \ \ c_{u_1}=\tfrac12-2\ve,\ \ \ \ c_{u_2}=-\tfrac12+\ve,\ \ \ \ c_w=\ve,
\]
crossing no poles. Then estimating the integral trivially establishes the lemma.
\end{proof}

Now we assemble $P_{\ol{M,N}}$ to evaluate $A_{\ol{O}}$. Let $\mathcal{A}_{\ol{O}}(\abgd)$ denote the main term of $A_{\ol{O}}(\abgd)$, and we deduce it in the following proposition.
\begin{proposition}\label{proolmathcalA}
We have
\begin{align}
\mathcal{A}_{\ol{O}}(\alpha,\beta,\gamma,\delta)=&\ol{\mathcal{P}}(\alpha,\beta,\gamma,\delta) +\ol{\mathcal{P}}(\beta,\alpha,\gamma,\delta)+\ol{\mathcal{P}}(\alpha,\beta,\delta,\gamma) +\ol{\mathcal{P}}(\beta,\alpha,\delta,\gamma) +O\left(q^{-\frac13+\ve}\r),\notag
\end{align}
where
\begin{align}\label{eqmathcalP}
\ol{\mathcal{P}}(\alpha,\beta,\gamma,\delta) &=\frac{\zeta_q(1-\alpha+\beta)\zeta_q(1-\gamma+\delta)} {\zeta_q(2-\alpha+\beta-\gamma+\delta)}q^{-\alpha-\gamma}\\
&\times\frac1{2\pi i} \int_{(\frac 14)}\frac{G(s)}{s}g_{\alpha,\beta,\gamma,\delta}(s)\frac{\Gamma(\frac12-\alpha-s) \Gamma(\frac12-\gamma-s)}{\Gamma(1-\alpha-\gamma-2s)} \notag\\
&\times \prod_{p\mid q}\left(1-\frac1{p^{1-(\alpha+\gamma+2s)}}\r)\zeta(\alpha+\gamma+2s) \zeta_q(1+\beta+\delta+2s) ds.\notag
\end{align}
\end{proposition}
\begin{proof}
It follows from \eqref{eqAOBMN}, \eqref{BBMN}, \eqref{eqBMN}, \eqref{eqwtCC}, and \eqref{eqCPPE} that
\begin{align}
\mathcal{A}_{\ol{O}}(\abgd)=\ol{P}(\abgd)+\ol{P}(\bagd)+\ol{P}(\abdg)+\ol{P}(\badg)\notag
\end{align}
with
\[
\ol{P}(\abgd)=\sum_{M\le N}P_{\ol{M,N}}(\abgd)+\sum_{M> N}P_{\ol{M,N}}(\abgd)
\]
and $P_{\ol{M,N}}(\abgd)=P_{\ol{N,M}}(\gdab)$ for $M>N$.

Applying this into \eqref{eqolPMN}, one notes that expressions of $P_{\ol{M,N}}$ for $M\le N$ and $M>N$ look different, so we can not recover the partition with them directly.
To resolve the problem, we eliminate the $w$-line of integral in \eqref{eqolPMN}. We move $c_s$ to $\frac 16$ followed by $c_w\rightarrow -\frac23+\ve$, and then a trivial estimate shows that the remained integral is $O\left(q^{-\frac13+\ve}\r)$. Note that the pole of $\zeta(1-\gamma+\delta+w)$ is cancelled by $G(\gamma-\delta)=0$,  and we have just crossed a pole at $w=0$. Calculating the residue gives
\begin{align}\label{eqPolMN}
P_{\ol{M,N}}&(\abgd)=\zeta_q(1-\alpha+\beta)\zeta(1-\gamma+\delta)\left(\frac1{2\pi i}\r)^3\int_{(c_s)}\int_{(c_{u_1})}\int_{(c_{u_2})} q^{-\alpha-\gamma-u_1-u_2} M^{u_1}N^{u_2}\\
&\times\prod_{p\mid q}\left(1-\frac1{p^{1-(\alpha+\gamma+2s+u_1+u_2)}}\r)\frac{\zeta(\alpha+\gamma+2s+u_1+u_2)\zeta_q(1+\beta+\delta+2s+u_1+u_2)} {\zeta_q(2-\alpha+\beta-\gamma+\delta)}\notag\\
&\times \frac{\Gamma(\tfrac12-\alpha-u_1-s)\Gamma(\frac12-\gamma-s-u_2)}{\Gamma(1-\alpha-\gamma-2s-u_1-u_2)}\frac{G(s)}{s} g(s) \wt{W}(u_1)\wt{W}(u_2)du_2du_1ds+O\left(q^{-\frac13+\ve}\r).\notag
\end{align}
It is easy to see from \eqref{eqPolMN} that
\[
P_{\ol{M,N}}(\abgd)=P_{\ol{N,M}}(\gdab)+O\left(q^{-\frac13+\ve}\r).
\]
In combination with the relationship $P_{\ol{M,N}}(\abgd)=P_{\ol{N,M}}(\gdab)$, one notes that \eqref{eqPolMN} also holds for $M>N$, that is to say, the two expressions of $P_{\ol{M,N}}$ for $M\le N$ and $M>N$ are unified in \eqref{eqPolMN}.

Now we come to recover the partition with \eqref{eqPolMN}, for which we need a formula due to Mellin transform
\begin{align}\label{eTqransform}
\sum_{M,N}\left(\frac 1{2\pi i}\r)^2\int_{(c_{u_1})}\int_{(c_{u_2})}F(u_1,u_2) \wt{W}(u_1)\wt{W}(u_2)du_1du_2=F(0,0)
\end{align}
for a `nice' function $F$ (see also the proof of \cite[Lemma 6.2]{You11}).
After the variable change $s\rightarrow s-\frac{u_1+u_2}2$ in \eqref{eqPolMN}, the zeta functions do not rely on $u_1$ and $u_2$ any more, and the product of gamma factors and $G$ becomes
\[
\frac{\Gamma\left(\tfrac12-\alpha-s-\frac{u_1-u_2}2\r)\Gamma(\frac12-\gamma-s+\frac{u_1-u_2}2)}{\Gamma(1-\alpha-\gamma-2s)}G\left(s-\frac{u_1+u_2}2\r),
\]
which decays rapidly in the imaginary direction of $u_1$ and $u_2$.
Thus, we sum over $M, N$ and recover the dyadic partition with \eqref{eTqransform} to get
\[
\ol{P}(\abgd)=\ol{\mathcal{P}}(\abgd)+O\left(q^{-\frac13+\ve}\r).
\]
This establishes the lemma.
\end{proof}

The treatment of $A_O$ is similar and easier. For $P_{M,N}$, there is a similar expression as in Lemma \ref{lemolPMN}, which also implies the upper bound \eqref{eqbmain}. Especially, the different gamma factors in $H$ do not bring any pole in calculating $P_{M,N}$, so we do not have to calculate a residue as \eqref{eqdualresidue}. In addition, one may recover the partition for $P_{M,N}$ directly since all $M\gg Nq^\ve$ terms are negligible. We use $\mathcal{A}_O(\abgd)$ for the main term of $A_O(\abgd)$.

\begin{proposition}\label{promathcalB}
We have
\begin{align}
\mathcal{A}_O(\alpha,\beta,\gamma,\delta)=&\mathcal{P}(\alpha,\beta,\gamma,\delta) +\mathcal{P}(\beta,\alpha,\gamma,\delta)+\mathcal{P}(\alpha,\beta,\delta,\gamma) +\mathcal{P}(\beta,\alpha,\delta,\gamma) +O\left(q^{-\frac12+\ve}\r),\notag
\end{align}
where
\begin{align}\label{eqmathcalP}
\mathcal{P}(\alpha,\beta,\gamma,\delta) &=\frac{\zeta_q(1-\alpha+\beta)\zeta_q(1-\gamma+\delta)} {\zeta_q(2-\alpha+\beta-\gamma+\delta)}q^{-\alpha-\gamma}\\
&\times\frac1{2\pi i} \int_{(\frac 14)}\frac{G(s)}{s}g_{\alpha,\beta,\gamma,\delta}(s)\frac{\Gamma(\frac12-\alpha-s) \Gamma(\alpha+\gamma+2s)}{\Gamma(\frac12+\gamma+s)} \notag\\
&\times \prod_{p\mid q}\left(1-\frac1{p^{1-(\alpha+\gamma+2s)}}\r)\zeta(\alpha+\gamma+2s) \zeta_q(1+\beta+\delta+2s) ds.\notag
\end{align}
\end{proposition}

One may note that the error term $q^{-\frac12+\ve}$ in Proposition \ref{promathcalB1} is sharper than the one in Proposition \ref{proolmathcalA}. This is also due to the different gamma factors in $H$, which allows our moving $c_w$ to the further left $-1+\ve$ when eliminating the $w$-integral.

We use $\mathcal{A}^*_O(\abgd)$ for the main term of $A^*_O(\abgd)$.
By the relationship between $A_O^*$ and $A_O$ in \eqref{A*OAO}, we switch the parameters to get
\begin{proposition}\label{promathcalB1}
We have
\begin{align}
\mathcal{A}_O^*(\alpha,\beta,\gamma,\delta)=&\mathcal{P}(\gamma,\delta,\alpha,\beta) +\mathcal{P}(\delta,\gamma,\alpha,\beta)+\mathcal{P}(\gamma,\delta,\beta,\alpha) +\mathcal{P}(\delta,\gamma,\beta,\alpha) +O\left(q^{-\frac12+\ve}\r),\notag
\end{align}
where $\mathcal{P}(\alpha,\beta,\gamma,\delta)$ is given by \eqref{eqmathcalP}.
\end{proposition}

\section{The main terms of the Theorem}\label{secmain}

In this section, we assemble all the main terms arising from both off-diagonal terms and diagonal terms to deduce the main terms of Theorem \ref{thmmains}, which proceeds much the same as \cite{You11}, and now we sketch it here.

At first, we combine $A_O$, $A^*_O$, and $A_{\ol{O}}$ to deduce the main contribution of the off-diagonal terms in $A_1$, which follows easily by \eqref{eqA_1} with Propositions \ref{proolmathcalA}, \ref{promathcalB}, and \ref{promathcalB1} that
\begin{align}\label{eqmoA_1}
Q(\abgd)+Q(\bagd)+Q(\abdg)+Q(\badg)+O(q^{-\frac13+\ve}),
\end{align}
where
\[
Q(\abgd)=\mathcal{P}(\abgd)+\mathcal{P}(\gdab)+\ol{\mathcal{P}}(\abgd).
\]

\begin{lemma}
We have
\begin{align}
Q(\abgd)=&\frac{\zeta_q(1-\alpha+\beta)\zeta_q(1-\gamma+\delta)}{\zeta_q(2-\alpha+\beta-\gamma+\delta)}\left(\frac q\pi\r)^{-\alpha-\gamma}\frac1{2\pi i}\int_{(\frac14)}\frac{G(s)}{s}\pi^{2s}g_{\abgd}(s)\notag\\
&\times\zeta_q(1-\alpha-\gamma-2s)\zeta_q(1+\beta+\delta+2s)\frac{\Gamma\left(\frac{\frac12-\alpha-s}{2}\r)}{\Gamma\left(\frac{\frac12+\alpha+s}{2}\r)} \frac{\Gamma\left(\frac{\frac12-\gamma-s}{2}\r)}{\Gamma\left(\frac{\frac12+\gamma+s}{2}\r)}ds.\notag
\end{align}
\end{lemma}

\begin{proof}
The lemma can be proved identically to \cite[Lemma 8.1]{You11}, and the only difference is to apply the functional equation
\[
\prod_{p\mid q}\left(1-\frac1{p^{1-(\alpha+\gamma+2s)}}\r)\zeta(\alpha+\gamma+2s)=\frac1{\pi^{\frac12-\alpha-\gamma-2s}} \frac{\Gamma\left(\frac{1-\alpha-\gamma-2s}{2}\r)}{\Gamma\left(\frac{\alpha+\gamma+2s}{2}\r)}\zeta_q(1-\alpha-\gamma-2s)
\]
in place of the functional equation of $\zeta(\alpha+\gamma+2s)$ in the original proof.
\end{proof}

Secondly, by the relationship \eqref{eqrA_1A_{-1}}, one can easily deduce the main contribution of the off-diagonal terms in $A_{-1}$ from \eqref{eqmoA_1},  by switching the signs of the shifts and multiplying by $X_{\alpha,\beta,\gamma,\delta}$, which is
\[
Q_{-}(\abgd)+Q_{-}(\bagd)+Q_{-}(\abdg)+Q_{-}(\badg)+O(q^{-\frac13+\ve})
\]
with
$Q_{-}(\abgd)=X_{\abgd}Q(-\gamma,-\delta,-\alpha,-\beta)$.

Now we combine the main terms of the off-diagonal terms in $A_1$ and $A_{-1}$. We couple $Q(\abgd)$ with $Q_{-}(\badg)$ and the other three pairs by switching the shifts, which follow easily from the following lemma.
\begin{lemma}
We have
\begin{align}
Q(\abgd)+Q_{-}(\badg)=X_{\alpha,\gamma}\frac{\zeta_q(1-\alpha+\beta)\zeta_q(1-\alpha-\gamma)\zeta_q(1+\beta+\delta)\zeta_q(1-\gamma+\delta)} {\zeta_q(2-\alpha+\beta-\gamma+\delta)}.\notag
\end{align}
\end{lemma}

\begin{proof}
The lemma can be proved identically to \cite[Lemma 8.3]{You11} with $\zeta$ replaced by $\zeta_q$ only.
\end{proof}

By combining these main terms with the diagonal terms in Lemma \ref{lemD}, we get the main term of $M$
which gives the main terms of Theorem \ref{thmmains}.

\section{Error terms with $M$ and $N$ close}\label{secError}
This section is devoted to bounding the error terms $E_{M,N}$ and $E_{\ol{M,N}}$. We will show only in detail the treatment of $E_{M,N}$ since the other one is similar except some technology variations. For ease of presentation, we set all the shifts $\alpha, \beta$, etc., equal to $0$, and one can easily generalize the arguments below to handle sufficiently small nonzero parameters.

\subsection{Initial treatment}\label{secEMN}
We first recall that $E_{M,N}$ is given by the right-hand side of \eqref{eqC_MN}, but with the contours of the integration at
\begin{align}\label{eqcontour}
c_s=c_w=c_v=\ve,\ \ \ \ c_{u_1}=c_{u_2}=0.
\end{align}
After applying the functional equation of Proposition \ref{lemgED2} to $\mathcal{D}_{q}$, we replace the sum over $a$ and $b$ with a factor $q^\ve$ and rename $a_1,b_1$ as $a, b$ for notational convenience in the following, and subsequently, $\chi'_0$ denotes the principal character modulo $ab$. Then we have
\begin{align}
E_{M,N}\ll q^\ve\mathop{\sum\sum\sum}_{\substack{a,b,d\mid q\\ (b,d)=1}}\mu^2(ab)(|E_+|+|E_-|),\notag
\end{align}
where\begin{align}
E_\pm&=\frac{\vp(d)}{\vp^*(q)}
\sum_{\substack{(r,b)=1 \\ d\mid r}}\sum_{(l,q)=1}\frac1{lab}\msum_{i(\bmod{a})} \msum_{j(\bmod{b})} \msum_{h(\bmod{l})} e\left(\frac{jr}{b}\r)e\left(\frac{hr}{l}\r)\notag\\
&\times\left(\frac1{2\pi i}\r)^5
\int_{(c_s)}\int_{(c_w)}\int_{(c_{u_1})}\int_{(c_{u_2})} \int_{(c_v)}\frac{q^{2s}M^{u_1}N^{u_2}} {r^{\frac12+s+u_2-v-\frac w2}l^{w}\left(lab\r)^{2s+2u_1+2v}}\notag\\
&\times D\Bigg(\tfrac12-s-u_1-v,0,\pm\frac{\ol{h_{i,j}}}{lab},\chi'_0\Bigg) H_\pm(s,u_1,u_2,v,w)dvdu_2du_1dwds, \notag
\end{align}
and where
\begin{align}\label{eqHpm}
H_\pm(s,u_1,u_2,v,w)=2(2\pi)^{-1+2s+2u_1+2v}\Gamma(\tfrac12-s-u_1-v)^2 H(s,u_1,u_2,v,w)S_\pm
\end{align}
with
$S_+=1$ and $S_-=\sin(\pi(s+u_1+v))$.
Since the growth of $S_-$ is canceled by $\Gamma(\tfrac12-s-u_1-v)^2$, $H_\pm$ also has a rapid decay as $H$ does.

\begin{notation}
In the expression of $E_{\pm}$, we should keep in mind that $a, b, d$ are divisors of $q$ with $a, b$ square free and $(ad,b)=1$, especially, there being $ab\le q$ and $bd\le q$. These conditions will be applied directly without any reminder in the following.
\end{notation}

We move $c_{u_1}$ to $-\frac12-3\ve$ and expand $D$ into an absolutely convergent Dirichlet series, then
\begin{align}\label{eqE_pm}
E_\pm=&\frac{\vp(d)}{\vp^*(q)}
\sum_{\substack{(r,b)=1 \\ d\mid r}}\sum_{(l,q)=1}\frac1{lab}\sum_{(m,ab)=1}\frac{d(m)\mathcal{S}_\pm(m,r)}{r^{\frac12}m^{\frac12}} \\
&\times\left(\frac1{2\pi i}\r)^5
\int_{(c_s)}\int_{(c_w)}\int_{(c_{u_1})}\int_{(c_{u_2})} \int_{(c_v)}\frac{q^{2s}M^{u_1}N^{u_2}m^{s+u_1+v}} {r^{s+u_2-v-\frac w2}l^w\left(lab\r)^{2s+2u_1+2v}}\notag\\
&\times H_\pm(s,u_1,u_2,v,w)dvdu_2du_1dwds \notag
\end{align}
with the exponential sum
\begin{align}\label{eqS_pm}
\mathcal{S}_\pm(m, r)=\msum_{i(\bmod a)} \msum_{j(\bmod b)} \msum_{h(\bmod l)} e\left(\frac{hr}{l}\r) e\left(\frac{jr}{b}\r)e\Bigg(\pm\frac{m\ol{h_{i,j}}}{lab}\Bigg).
\end{align}
With a negligible error, we may restrict the $m$-sum in $E_{\pm}$ to $m\le M^{-1}Na^2b^2q^\ve$, which is always assumed in the rest of this section.
To see this, we move the contours in $E_{\pm}$ to
$c_s=2,~c_{u_1}=-A,~c_{u_2}=A,~ c_w=2A,~ c_v=\ve$
for large $A$, crossing no poles. Then a trivial estimate shows that the contribution of all terms with $m\gg M^{-1}Na^2b^2q^\ve$ is negligible.

We distinguish the Kloosterman sum from $\mathcal{S}_\pm(m, r)$ in the following lemma.

\begin{lemma}\label{lemes}
Let $a, b$, and $l$ be three integers coprime with each other. We have
\begin{align}
\mathcal{S}_\pm(m, r)=c_{a}(m)S(r\ol{a}^2,\pm m;lb).\notag
\end{align}
\end{lemma}
\begin{proof}
By \eqref{eqhij}, we apply Lemma \ref{lemmaspe(x)} to split the last exponential function as
\begin{align}
e\Bigg(\pm\frac{m\ol{h_{i,j}}}{lab}\Bigg) = e\Bigg(\pm\frac{m\ol{i}(\ol{lb})^2}{a}\Bigg) e\Bigg(\pm\frac{m\ol{(hb+jl)}\ol{a}^2}{lb}\Bigg).\notag
\end{align}
Inserting this into \eqref{eqS_pm}, we sum over $i$ first, which turns out to be a Ramanujan sum $c_{a}(m)$. Besides, we apply the Chinese Remainder Theorem to write the sums over $j$ and $h$ as a Kloosterman sum that
\[
\mathop{\sum\nolimits^*}_{j(\bmod {b})}\mathop{\sum\nolimits^*}_{h(\bmod{l})} e\left(\frac{jr}{b}\r)e\left(\frac{hr}{l}\r)e\Bigg(\pm\frac{m\ol{(hb+jl)}\ol{a}^2}{lb}\Bigg)=S(r\ol{a}^2,\pm m;lb).
\]
This establishes the lemma.
\end{proof}

Applying Lemma \ref{lemes} in \eqref{eqE_pm}, we have
\begin{align}
E_\pm=&\frac{\vp(d)}{\vp^*(q)}
\sum_{\substack{(r,b)=1 \\ d\mid r}}\sum_{\substack{m\le M^{-1}Na^2b^2q^\ve\\ (m,ab)=1}} \frac{d(m)c_{a}(m)}{r^{\frac12}m^{\frac12}}\sum_{(l,q)=1}\frac{S(r\ol{a}^2,\pm m;lb)}{lab} \notag\\
&\times\left(\frac1{2\pi i}\r)^5
\int_{(c_s)}\int_{(c_w)}\int_{(c_{u_1})}\int_{(c_{u_2})} \int_{(c_v)}\frac{q^{2s}M^{u_1}N^{u_2}m^{s+u_1+v}} {r^{s+u_2-v-\frac w2}l^w\left(lab\r)^{2s+2u_1+2v}}\notag\\
&\times H_\pm(s,u_1,u_2,v,w)dvdu_2du_1dwds. \notag
\end{align}
Note that
$c_{a}(m)=\mu(a)$
for $(m,a)=1$.\footnote{The coprime condition $(m,a)=1$ benefits from our kicking out nonessential terms in the deduction of the functional equation for $\mathcal{D}_q$. If not the case, an extra factor from $c_a(m)$ will lead to an obstruction.}
We split the coprime condition $(l,q)=1$ into $(l,a)=1$ and $(l,q_{a})=1$ and write $(l,q_{a})=1$ in terms of M\"{o}bius function.
It follows that
\begin{align}
E_\pm\ll \sum_{c\mid q_a}|\mathcal{E}_{\pm}|\notag
\end{align}
with
\begin{align}
\mathcal{E}_{\pm}=&\frac{\vp(d)}{\vp^*(q)}
\sum_{\substack{(r,b)=1 \\ d\mid r}} \sum_{\substack{m\le M^{-1}Na^2b^2q^\ve\\ (m,ab)=1}} \frac{d(m)}{r^{\frac12}m^{\frac12}} \sum_{(l,a)=1}\frac{S(r\ol{a}^2,\pm m;lbc)}{labc}\notag\\
&\times\left(\frac1{2\pi i}\r)^5
\int_{(c_s)}\int_{(c_w)}\int_{(c_{u_1})}\int_{(c_{u_2})} \int_{(c_v)}\frac{q^{2s}M^{u_1}N^{u_2}m^{s+u_1+v}} {r^{s+u_2-v-\frac w2}(lc)^w\left(labc\r)^{2s+2u_1+2v}}\notag\\
&\times H_\pm(s,u_1,u_2,v,w)dvdu_2du_1dwds.\notag
\end{align}
We should keep in mind that $a,b,c$ are square-free divisors of $q$, satisfying $(a,bc)=1$.

Let $\mathcal{E}_{\pm}(R,M^*)$ be the same expression as $\mathcal{E}_{\pm}$ but with $r$ and $m$ restricted to respective dyadic segments $R\le r\le 2R$ and $M^*\le m\le 2M^*$  with $M^*\ll M^{-1}Na^2b^2q^\ve$.

\subsection{The spectral decomposition}
The Kloostermn sum $S(r\ol{a}^2,\pm m;lbc)$ is complicated to handle, and we should make an evolution before making the spectral decomposition.
We apply \eqref{eqSS} with $Q=a^2bc$, $\tau=a^2$, and $s=bc$, which gives
\[
S(r\ol{a}^2,\pm m;lbc)=e\Bigg(\mp m\frac{\ol{bc}}{a^2}\Bigg)S_{\infty, 1/bc}(r, \pm m; \gammaup).
\]
Note that $S_{\infty, 1/bc}(r,\pm m; \gammaup)$ is defined if and only if, $\gammaup=labc$ with integer $l$ coprime with $a$, which is to say, the sum over $l$ in $\mathcal{E}_{\pm}$ is equal to the sum over $\gammaup$.

We replace the sum over $l$ with $\gammaup$ in $\mathcal{E}_{\pm}(R,M^*)$, which is
\begin{align}\label{eqERM}
\mathcal{E}_{\pm}(R,M^*)=&\frac{\vp(d)}{\vp^*(q)}
\sum_{\substack{R\le r\le 2R\\ (r,b)=1, d\mid r}}\sum_{\substack{M^*\le m\le2M^*\\ (m,ab)=1}} \frac{e\left(\mp m\frac{\ol{bc}}{a^2}\r)d(m)}{r^{\frac12}m^{\frac12}} \Sigma_{\pm}(r,m),
\end{align}
where
\[
\Sigma_{\pm}(r,m)=\sum_{\gammaup}^{\Gamma}\frac{S_{\infty, 1/bc}(r, \pm m; \gammaup)}{\gammaup}\phi_{\pm}\left(\frac{4\pi\sqrt{mr}}{\gammaup}\r)
\]
with
\begin{align}
\phi_{\pm}(x)=\left(\frac1{2\pi i}\r)^5
\int\cdots \int\frac{a^wb^wq^{2s}M^{u_1}N^{u_2}} {r^{2s+u_1+u_2}m^{\frac w2}} \left(\frac{x}{4\pi}\r)^{2s+2u_1+2v+w}H_\pm(s,u_1,u_2,v,w)dvdu_2du_1dwds.\notag
\end{align}
Here the contours of the integration are as in \eqref{eqcontour}. By taking $c_s=\frac12-2\ve$, one can easily check that $\phi_{\pm}(0)=0$. Also, taking  $c_{u_1}=-A$ for large $A$ gives $\phi_{\pm}^{(j)}(x)\ll (1+x)^{-2-\ve}$, $j=0,1,2$.

We apply the Kuznetsov formula with $\mathfrak{a}=\infty$ and $\mathfrak{b}=1/bc$ to get
\begin{align}
\Sigma_{+}(r,m)=&\sum_{\substack{k\ge2\\ k~ \text{even}}}\sum_{f\in\mathcal{B}_k(Q)}\tilde{\phi}_+(k)\Gamma(k)\sqrt{rm}\ol{\rho}_{f}(r)\rho_f(\mathfrak{b},m)\notag\\
&+\sum_{f\in\mathcal{B}(Q)}\hat{\phi}_+(\kappa_f)\frac{\sqrt{rm}}{\cosh(\pi \kappa_f)}\ol{\rho}_f(r)\rho_f(\mathfrak{b},m)\notag\\ &+\frac1{4\pi}\sum_{\mathfrak{c}}\int_{-\infty}^{\infty}\hat{\phi}_+(\kappa)\frac{\sqrt{rm}}{\cosh(\pi \kappa)}\ol{\rho}_{\mathfrak{c}}(r,\kappa)\rho_{\mathfrak{b},\mathfrak{c}}(m,\kappa)d\kappa\notag
\end{align}
and
\begin{align}
\Sigma_{-}(r,m)=&\sum_{f\in\mathcal{B}(Q)}\breve{\phi}_-(\kappa_f)\frac{\sqrt{rm}}{\cosh(\pi \kappa_f)}\ol{\rho}_f(r)\rho_f(\mathfrak{b},-m)\notag\\ &+\frac1{4\pi}\sum_{\mathfrak{c}}\int_{-\infty}^{\infty}\breve{\phi}_-(\kappa)\frac{\sqrt{rm}}{\cosh(\pi \kappa)}\ol{\rho}_{\mathfrak{c}}(r,\kappa)\rho_{\mathfrak{b},\mathfrak{c}}(-m,\kappa)d\kappa.\notag
\end{align}
Applying these into \eqref{eqERM}, we rewrite
\[
\mathcal{E}_{\pm}(R,M^*)=\mathcal{E}_{h\pm}(R,M^*)+\mathcal{E}_{m\pm}(R,M^*)+\mathcal{E}_{c\pm}(R,M^*)
\]
to correspond to the holomorphic forms, the Maass forms, and the Eisenstein series. Obviously, there is $\mathcal{E}_{h-}(R,M^*)=0$.

\subsection{Integral transforms}

By integral transforms of Bassel functions as in \cite[Section 9.3]{You11}, we have
\begin{align}
\tilde{\phi}_+(k)=\left(\frac1{2\pi i}\r)^5\int\cdots\int\frac{a^wb^wq^{2s}M^{u_1}N^{u_2}}{r^{2s+u_1+u_2} m^{\frac w2}}\tilde{H}_+(s,u_1,u_2,v,w,k)dvdu_1du_2dwds,\notag
\end{align}
where
\begin{align}
\tilde{H}_+(s,u_1,u_2,v,w,k)=&\frac{\Gamma(\frac{k-1}2+s+u_1+v+\tfrac w2)}{\Gamma(\frac{k+1}2-s-u_1-v-\tfrac w2)}\Gamma(\tfrac12-s-u_1-v)^2\notag\\
&\times \frac{\Gamma(v)\Gamma(\frac12+s+u_2-v-\frac w2)}{\Gamma(\frac12+s+u_2-\frac w2)} \frac{G(s)G(w)}{sw}g(s)\wt{W}(u_1)\wt{W}(u_2)\zeta(1+w)c^*\notag
\end{align}
with $c^*$ being some bounded factors like powers of $2$, $\pi$, etc., which do not affect the convergence of the integral. Also, there exist
\begin{align}
\hat{\phi}_+(\kappa)=\left(\frac1{2\pi i}\r)^5\int\cdots\int\frac{a^wb^wq^{2s}M^{u_1}N^{u_2}}{r^{2s+u_1+u_2} m^{\frac w2}}\hat{H}_+(s,u_1,u_2,v,w,\kappa)dvdu_1du_2dwdsp\notag
\end{align}
with
\begin{align}
\hat{H}_+(s,u_1,u_2,v,w,\kappa)=&\cos(\pi(s+u_1+v+\tfrac w2))\Gamma(s+u_1+v+\tfrac w2+i\kappa)\Gamma(s+u_1+v+\tfrac w2-i\kappa)\notag\\
&\times \Gamma(\tfrac12-s-u_1-v)^2\frac{\Gamma(v)\Gamma(\frac12+s+u_2-v-\frac w2)}{\Gamma(\frac12+s+u_2-\frac w2)}\notag\\
&\times\frac{G(s)G(w)}{sw}g(s)\wt{W}(u_1)\wt{W}(u_2)\zeta(1+w)c^*,\notag
\end{align}
and
\begin{align}\label{eqphi-}
\breve{\phi}_-(\kappa)=\left(\frac1{2\pi i}\r)^5\int\cdots\int\frac{a^wb^wq^{2s}M^{u_1}N^{u_2}}{r^{2s+u_1+u_2} m^{\frac w2}}\cosh(\pi\kappa)\breve{H}_-(s,u_1,u_2,v,w,\kappa)dvdu_1du_2dwds
\end{align}
with
\begin{align}
\breve{H}_-(s,u_1,u_2,v,w,\kappa)=&\Gamma(s+u_1+v+\tfrac w2+i\kappa)\Gamma(s+u_1+v+\tfrac w2-i\kappa)\notag\\
&\times \sin(\pi(s+u_1+v))\Gamma(\tfrac12-s-u_1-v)^2\frac{\Gamma(v)\Gamma(\frac12+s+u_2-v-\frac w2)}{\Gamma(\frac12+s+u_2-\frac w2)}\notag\\
&\times\frac{G(s)G(w)}{sw}g(s)\wt{W}(u_1)\wt{W}(u_2)\zeta(1+w)c^*.\notag
\end{align}

\subsection{The continuous spectrum}
The treatments of $\mathcal{E}_{h+}$, $\mathcal{E}_{m\pm}$, and $\mathcal{E}_{c\pm}$ are based on spectral large sieve inequalities given in Lemma \ref{lemsls}, and follow in the same way with slight variations. The variations are due to their different Fourier coefficients and gamma factors. We show in detail for $\mathcal{E}_{c\pm}$ and point out necessary variations for other cases.

\begin{proposition}\label{proE}
We have
\begin{align}\label{eqproE}
\mathcal{E}_{c\pm}(R,M^*)\ll q^{-\frac12+\ve}\left(\frac NM\r)^{\frac12}R^{-\ve}.
\end{align}
\end{proposition}

\begin{proof}
We only show in detail for $\mathcal{E}_{c-}$ since the case of $\mathcal{E}_{c+}$ is similar.
We recall that
\begin{align}
\mathcal{E}_{c-}(R,M^*)=&\frac{\vp(d)}{\vp^*(q)}
\sum_{\substack{R\le r\le 2R\\ (r,b)=1, d\mid r}}\sum_{\substack{M^*\le m\le2M^*\\ (m,ab)=1}} \frac{e\left(m\frac{\ol{bc}}{a^2}\r)d(m)}{r^{\frac12}m^{\frac12}} \notag\\
&\times\frac1{4\pi}\sum_{\mathfrak{c}}\int_{-\infty}^{\infty}\breve{\phi}_-(\kappa)\frac{\sqrt{rm}}{\cosh(\pi \kappa)}\ol{\rho}_{\mathfrak{c}}(r,\kappa)\rho_{\mathfrak{b},\mathfrak{c}}(-m,\kappa)d\kappa.\notag
\end{align}
Let $\mathcal{E}_{c-}(R,M^*;K)$ be the same expression as $\mathcal{E}_{c-}(R,M^*)$ but with the integral over $\kappa$ restricted to $[K,2K]$,
and then it follows by the integral transform \eqref{eqphi-} that
\begin{align}\label{eqformulaEc-}
\mathcal{E}_{c-}&(R,M^*;K)\ll\frac{\vp(d)}{\vp^*(q)}
\sum_{\substack{R\le r\le 2R\\ (r,b)=1, d\mid r}}\sum_{\substack{M^*\le m\le2M^*\\ (m,ab)=1}} \frac{e\left(m\frac{\ol{bc}}{a^2}\r)d(m)}{r^{\frac12}m^{\frac12}} \sum_{\mathfrak{c}}\int_{K}^{2K}\frac{\sqrt{rm}}{\cosh(\pi \kappa)}\ol{\rho}_{\mathfrak{c}}(r,\kappa)\rho_{\mathfrak{b},\mathfrak{c}}(-m,\kappa)\\
&\ \ \ \times \left(\frac1{2\pi i}\r)^5\int\cdots\int\frac{a^w b^w q^{2s}M^{u_1}N^{u_2}}{r^{2s+u_1+u_2} m^{\frac w2}}\cosh(\pi\kappa)\breve{H}_-(s,u_1,u_2,v,w,\kappa)dvdu_1du_2dwds d\kappa.\notag
\end{align}
The rapid decay of $\breve{H}_-$ allows for truncation of all integrals at height $(qK)^\ve$ with a negligible error except the integral over $\kappa$.

We apply Stirling's approximation to see that
\begin{align}\label{eqboundGamma}
\cosh(\pi\kappa)\Gamma(s+u_1+v+\tfrac w2+i\kappa)\Gamma(s+u_1+v+\tfrac w2-i\kappa)\ll q^\ve K^{-1+2c_s+2c_{u_1}+2c_v+c_w}.
\end{align}
We now come to consider the spectral sum, and
by Cauchy-Schwarz inequality
\begin{align}\label{eqspectralsum}
\text{Spec}&=\sum_{\substack{R\le r\le 2R\\ (r,b)=1, d\mid r}}\sum_{\substack{M^*\le m\le2M^*\\ (m,ab)=1}} \frac{e\left(m\frac{\ol{bc}}{a^2}\r)d(m)}{r^{\frac12+2s+u_1+u_2}m^{\frac12+\frac w2}} \sum_{\mathfrak{c}}\int_{K}^{2K}\frac{\sqrt{rm}}{\cosh(\pi \kappa)}\ol{\rho}_{\mathfrak{c}}(r,\kappa)\rho_{\mathfrak{b},\mathfrak{c}}(-m,\kappa)d\kappa\\
&\ll\Bigg(\sum_{\mathfrak{c}}\int_{K}^{2K}\frac{1}{\cosh(\pi \kappa)}\Bigg|\sum_{\substack{R\le r\le 2R\\ (r,b)=1, d\mid r}}\frac1{r^{\frac12+2s+u_1+u_2}}\sqrt{r}\ol{\rho}_{\mathfrak{c}}(r,\kappa)\Bigg|^2d\kappa\Bigg)^{\frac12}\notag\\
&\ \ \ \ \times\Bigg(\sum_{\mathfrak{c}}\int_{K}^{2K}\frac{1}{\cosh(\pi \kappa)}\Bigg|\sum_{\substack{M^*\le m\le2M^*\\ (m,ab)=1}} \frac{e\left(m\frac{\ol{bc}}{a^2}\r)d(m)}{m^{\frac12+\frac w2}}\sqrt{m}\rho_{\mathfrak{b},\mathfrak{c}}(-m,\kappa)\Bigg|^2\Bigg)^{\frac12}.\notag
\end{align}

In \eqref{eqspectralsum}, we bound the second factor first, which follows directly by the large sieve inequality of Lemma \ref{lemsls} that
\begin{align}
\ll \left(K^2+\frac{M^*}{Q}\r)^{\frac12}\Bigg(\sum_{\substack{M^*\le m\le2M^*\\ (m,ab)=1}}\frac1{m^{1+\ve}}\Bigg)^{\frac12}\ll \left(K+\frac {N^{\frac12}}{M^{\frac12}}\r)b^{\frac12}q^\ve\notag
\end{align}
since $\muup(\mb)=Q^{-1}$, $M^*\ll M^{-1}Na^2b^2q^\ve$, and $Q=a^2bc$.

The treatment of the first factor needs more elaboration. Let $d^{\sharp}$ always denote an integer, which owns the same distinct prime factors as $d$ and meets $d\mid d^{\sharp}$.
With the change $r\rightarrow d^{\sharp}r'$, we can split the $r$ sum into two sums over both $d^{\sharp}$ and $r'$ with $(r',d)=1$.
That is to say,
\[
\sum_{\substack{R\le r\le 2R\\ (r,b)=1, d\mid r}}=\sum_{\substack{d^{\sharp}\le 2R\\ (d^{\sharp},b)=1}}\sum_{\substack{R/d^\sharp\le r'\le 2R/d^\sharp\\ (r',bd)=1}}.
\]
The number of $d^{\sharp}$ is bounded by $O(R^\ve)$ for any given $\ve>0$, which one may prove as follow:
\[
\sum_{d^{\sharp}\le 2R}\ll \sum_{d^{\sharp}}\left(\frac{R^\ve}{d^{\sharp}}\r)^{\ve}\ll R^\ve d^{-\ve}\prod_{p\mid d}(1-p^{-\ve})^{-1}\ll R^\ve.
\]
Then by \eqref{eqErho}, the first factor is at most
\begin{align}\label{eqsumr}
&\sum_{d^{\sharp}\le 2R}\Bigg(\sum_{\mathfrak{c}}\int_{K}^{2K}\frac{1}{\cosh(\pi \kappa)}\Bigg|\sum_{\substack{R\le r=r'd^\sharp\le 2R\\ (r',bd)=1}}\frac1{r^{\frac12+2s+u_1+u_2}}\sqrt{r}\ol{\rho}_{\mathfrak{c}}(r,\kappa)\Bigg|^2d\kappa\Bigg)^{\frac12}\\
&\ll q^\ve\sum_{d^{\sharp}\le 2R}\sum_{d_1\mid (d^\sharp,Q)}\Bigg(\sum_{\mathfrak{c}}\int_{K}^{2K}\frac{1}{\cosh(\pi \kappa)}\Bigg|\sum_{\substack{R/d^\sharp\le r'\le 2R/d^\sharp\\ (r',bd)=1}}\frac1{(r'd^\sharp)^{\frac12+2s+u_1+u_2}}\sqrt{d_1r'}\ol{\rho}_{\mathfrak{c}}(d_1r',\kappa)\Bigg|^2d\kappa\Bigg)^{\frac12}.\notag
\end{align}
 After applying the large sieve inequality of Lemma \ref{lemsls}, we find that it is bounded by
\begin{align}\label{eqrddsharp}
&\ll q^\ve\sum_{d^{\sharp}\le 2R}\sum_{d_1\mid (d^\sharp,Q)}\left(K^2+\frac{d_1R}{d^\sharp Q}\r)^{\frac12}\Bigg(\sum_{\substack{R/d^\sharp\le r'\le 2R/d^\sharp\\ (r',bd)=1}}\frac1{(r'd^\sharp)^{1+4c_s+2c_{u_1}+2c_{u_2}}}\Bigg)^{\frac12}\\
&\ll q^\ve\sum_{d^{\sharp}\le 2R}\sum_{d_1\mid (d^\sharp,Q)}\left(K^2+\frac{d_1R}{d^\sharp Q}\r)^{\frac12}\Bigg(\frac{1}{d^\sharp R^{4c_s+2c_{u_1}+2c_{u_2}}}\Bigg)^{\frac12}\notag\\
&\ll \frac{q^\ve}{R^{2c_s+c_{u_1}+c_{u_2}-\ve}}\Bigg(\frac{K}{d^\frac12}+\frac{R^{\frac12}}{dQ^{\frac12}}\sum_{d_1\mid (d,Q)}d_1^{\frac12}\Bigg).\notag
\end{align}
In the final step, we have replaced $d^\sharp$ by $d$ with its quantity $\ll R^\ve$, observing that $d^\sharp\ge d$ has the same prime factors as $d$.

In conclusion, we bound the spectral sum in \eqref{eqspectralsum} by
\begin{align}\label{eqboundss}
\text{Spec}&\ll \frac{q^\ve}{R^{2c_s+c_{u_1}+c_{u_2}-\ve}}\Bigg(K+\frac {N^{\frac12}}{M^{\frac12}}\Bigg)\Bigg(\frac{K}{d^\frac12}+\frac{R^{\frac12}}{d Q^{\frac12}}\sum_{d_1\mid (d,Q)}d_1^{\frac12}\Bigg)b^{\frac12}\\
&\ll \frac{q^{\frac12+\ve}}{dR^{2c_s+c_{u_1}+c_{u_2}-\ve}}\Bigg(K+\frac {N^{\frac12}}{M^{\frac12}}\Bigg) \Bigg(K+q^{-\frac12}R^{\frac12}\Bigg).\notag
\end{align}
Here we have applied two inequalities $d_1 b\le Q$ and $db\le q$, which follow directly from the facts
$
d_1\mid Q,~b\mid Q,~(d_1,b)=1$,
and
$d\mid q,~ b\mid q,~(d,b)=1$.

We apply \eqref{eqboundGamma} and \eqref{eqboundss} into \eqref{eqformulaEc-} to see
\begin{align}\label{eqER}
\mathcal{E}_{c-}(R,M^*;K)
\ll q^{-\frac12+2c_s+\ve}M^{c_{u_1}}N^{c_{u_2}}K^{-1+2c_s+2c_{u_1}+3\ve}\Bigg(K+\frac {N^{\frac12}}{M^{\frac12}}\Bigg)\left(K+q^{-\frac12}R^{\frac12}\r)R^{-2c_s-c_{u_1}-c_{u_2}+\ve}.
\end{align}
Observing the poles of $\breve{H}_-$, we may bound $\mathcal{E}_{c-}(R,M^*;K)$ with appropriate values of $c_v$, $c_w$, $c_s$, $c_{u_1}$, $c_{u_2}$.

When $R\le K^2q$, a direct calculation of \eqref{eqER} with $c_v=c_w=c_s=\ve$ and $c_{u_1}=-c_{u_2}=\tfrac 12-3\ve$ for $K\ll M^{-\frac12}N^{\frac12}q^\ve$, $c_{u_1}=-c_{u_2}=-\tfrac 12-3\ve$ for $K\gg M^{-\frac12}N^{\frac12}q^\ve$ shows
\begin{align}
\mathcal{E}_{c-}(R,M^*;K)\ll q^{-\frac12+\ve}\left(\frac NM\r)^{\frac12}R^{-\ve}K^{-\ve}.\notag
\end{align}

When $R\ge K^2q$, one notes from \eqref{eqER} with $c_v=c_w=\ve$, $c_s=\tfrac 14+\ve$, and $c_{u_1}=-c_{u_2}=\tfrac 14-3\ve$ for $K\ll M^{-\frac12}N^{\frac12}q^\ve$, $c_{u_1}=-c_{u_2}=-\tfrac 14-3\ve$ for $K\gg M^{-\frac12}N^{\frac12}q^\ve$ that
\begin{align}
\mathcal{E}_{c-}(R,M^*;K)\ll q^{-\frac12+\ve}\left(\frac NM\r)^{\frac14}R^{-\ve}K^{-\ve}.\notag
\end{align}
Then summing over $K$ establishes the proposition.
\end{proof}

\subsection{The other two spectrums}\label{secOTS}
We now come to $\mathcal{E}_{h+}(R,M^*;K)$. Note that
\begin{align}\label{eqHG}
\frac{\Gamma(\frac{k-1}2+s+u_1+v+\tfrac w2)}{\Gamma(\frac{k+1}2-s-u_1-v-\tfrac w2)}\ll K^{-1+2c_s+2c_{u_1}+2c_v+c_w}
\end{align}
for small $s, u_1, v, w$,
so the same treatment as the Eisenstein spectrum gives the same bound for the holomorphic spectrum. We bound the Maass spectrum in the following.
\begin{proposition}\label{proM}
We have
\begin{align}
\mathcal{E}_{m\pm}(R,M^*)\ll q^{-\frac12+\theta+\ve}\left(\frac NM\r)^{\frac12}R^{-\ve}.
\end{align}
\end{proposition}
\begin{proof}
The proposition follows by the same argument as the Eisenstein spectrum, but with a replacement of \eqref{eqErho} by \eqref{eqMrho} to separate $d^\sharp$ from the Fourier coefficients. We can bound the extra factor $\lambda_{f^*}$ arising from the replacement with the well-known bound
$\lambda_{f^*}(n)\ll n^{\theta+\ve}$.
More precisely, the following estimate due to \eqref{eqMrho} is applied in the analog of \eqref{eqsumr} to replace \eqref{eqErho}
\[
\sum_{\substack{R\le r\le 2R\\ r=r'd^\sharp\\ (r',bd)=1}}\frac1{r^{\frac12+2s+u_1+u_2}}\sqrt{r}\ol{\rho}_f(r)\ll {d^{\sharp}}^{\theta+\ve}\sum_{d_1\mid (d^{\sharp},Q)}\Bigg|\sum_{\substack{R/d^\sharp\le r'\le 2R/d^\sharp\\ (r',bd)=1}}\frac1{(r'd^\sharp)^{\frac12+2s+u_1+u_2}}\sqrt{d_1r'}\rho_f(d_1r')\Bigg|.
\]
Sine $\theta=\frac7{64}<\frac12$, we can also replace $d^\sharp$ by $d$ in the analog of \eqref{eqrddsharp}. Then the extra factor ${d^\sharp}^{\theta+\ve}$ contributes at most a factor $q^{\theta+\ve}$ to $\mathcal{E}_{m\pm}(R,M^*)$ at last.
\end{proof}

At last, one obtains the bound in Theorem \ref{thmEMN} by summing over $R$ and $M^*$ for these bounds of the holomorphic spectrum, the Maass spectrum, and the Eisenstein spectrum.

\section{Off-diagonal terms with $M$ and $N$ far away}\label{secfar}
This section is devoted to bounding $B_{M,N}$ and $B_{\ol{M,N}}$ with $M, N$ far away from each other. We lay focus on $B_{M,N}$ since the case for $B_{\ol{M,N}}$ is identical.
For notational convenience, we set all shifts to zero as before, and the arguments extend easily to nonzero parameter values.

Let $\eta=\frac1{14}-\frac 3{7}\theta$. We now set
\[
M=q^{\mu},\ \ \ \ N=q^{\nu},
\]
so by \eqref{eqEMN} and \eqref{eqtrivalbound}, proving Theorem \ref{thmmains} is to show
\begin{align}
B_{M,N}\ll q^{-\eta+\ve}\notag
\end{align}
for
\begin{align}\label{eq1munu}
2-2\eta\le\mu+\nu\le 2,\ \ \ \ \ \  1-2\theta-2\eta\le\nu-\mu.
\end{align}

We recall that
\begin{align}
B_{M,N}=\frac1{\vp^*(q)}\sum_{d\mid q}\vp(d)\mu\left(\frac qd\r)\mathop{\sum\sum}_{\substack{m\equiv n(\bmod d)\\ (mn,q)=1}}\frac{d(m)d(n)}{m^{\frac12}n^{\frac12}}V\left(\frac{mn}{q^2}\r) W\left(\frac mM\r)W\left(\frac nN\r),\notag
\end{align}
which after the definition of $V$ being applied evolves into
\begin{align}
B_{M,N}=&\frac1{2\pi i}\int_{(\ve)}\left(\frac{q^2}{MN}\r)^{s}\frac{G(s)}{s}g(s)\notag\\
&\times\frac1{\vp^*(q)\sqrt{MN}}\sum_{d\mid q}\vp(d)\mu\left(\frac qd\r)\mathop{\sum\sum}_{\substack{m\equiv n(\bmod d)\\ (mn,q)=1}}d(m)d(n) W_s\left(\frac mM\r)W_s\left(\frac nN\r),\notag
\end{align}
where $W_s(x)=x^{\frac12+s}W(x)$. Since $G(s)g(s)$ decays rapidly in the imaginary direction, we omit the effect of $s$ to get
\begin{align}
B_{M,N}\ll\frac1{\vp^*(q)\sqrt{MN}}\sum_{d\mid q}\vp(d)\mu\left(\frac qd\r)\mathop{\sum\sum}_{\substack{m\equiv n(\bmod d)\\ (mn,q)=1}}d(m)d(n) W\left(\frac mM\r)W\left(\frac nN\r).\notag
\end{align}

We write $d(n)=\sum\limits_{n_1n_2=n}1$ and apply the dyadic partition of unity to both $n_1$ and $n_2$ with $n_1\as N_1$, $n_2\as N_2$, and $N_1N_2\as N$. Without loss of generality, we assume $N_2\ge N_1$. Now we have reduced the problem to bounding $B_{M,N_1,N_2}$, defined as
\begin{align}
B_{M,N_1,N_2}=&\frac1{\vp^*(q)\sqrt{MN}}\sum_{d\mid q}\vp(d)\mu\left(\frac qd\r)\mathop{\sum\sum\sum}_{\substack{m\equiv n_1n_2(\bmod d)\\ (mn_1n_2,q)=1}}d(m) W\left(\frac mM\r)W\left(\frac {n_1}{N_1}\r)W\left(\frac {n_2}{N_2}\r)W\left(\frac {n_1n_2}{N_1N_2}\r).\notag
\end{align}
We set
\[
N_1=q^{\nu_1},\ \ \ \ \ \ N_2=q^{\nu_2},
\]
so that
\begin{align}\label{eqnu1nu2u}
\nu=\nu_1+\nu_2,\ \ \ \ \ \nu_1\le \nu_2.
\end{align}

After applying Mellin transform to $W\left(\frac{n_1n_2}{N_1N_2}\r)$, one may remove its effect as $V$. Then it follows that
\begin{align}
B_{M,N_1,N_2}\ll\frac1{\sqrt{MN}}\sum_{(m,q)=1}d(m)\Delta_{N_1,N_2}(m)W\left(\frac mM\r)\notag
\end{align}
with
\begin{align}
\Delta_{N_1,N_2}(m)=\frac1{\vp^*(q)}\sum_{d\mid q}\vp(d)\mu\left(\frac qd\r) \sum_{(n_1,q)=1}\sum_{\substack{n_2\equiv m\ol{n}_1(\bmod d)\\ (n_2,q)=1}}W\left(\frac {n_1}{N_1}\r) W\left(\frac {n_2}{N_2}\r).\notag
\end{align}
It is obvious that
\begin{align}
\begin{array}{c}
n_2\equiv m\ol{n}_1 \pmod d\\
(n_2,q)=1
\end{array}\Longleftrightarrow
\begin{array}{c}
n_2\equiv m\ol{n}_1 \pmod d\\
(n_2,q_d)=1
\end{array}.\notag
\end{align}
After applying this, we write $(n_2,q_d)=1$ in terms of M\"{o}bius function to get
\begin{align}
\Delta_{N_1,N_2}(m)=&\frac1{\vp^*(q)}\sum_{d\mid q}\vp(d)\mu\left(\frac qd\r) \sum_{a\mid q_d}\mu(a)\sum_{(n_1,q)=1} \sum_{\substack{n_2\equiv m\ol{n}_1(\bmod d)\\ n_2\equiv0(\bmod a)}}  W\left(\frac {n_1}{N_1}\r)W\left(\frac {n_2}{N_2}\r).\notag
\end{align}
Applying Poisson summation formula to $n_2$, we have
\begin{align}
\Delta_{N_1,N_2}(m)=&\frac{N_2}{\vp^*(q)}\sum_{d\mid q}\frac{\vp(d)}{d}\mu\left(\frac qd\r) \sum_{a\mid q_d}\frac{\mu(a)}{a}\sum_{(n_1,q)=1}\sum_{h} e\left(\frac{hm\ol{a}\ol{n}_1}{d}\r) W\left(\frac {n_1}{N_1}\r)\wh{W}\left(\frac {h}{ad/N_2}\r),\notag
\end{align}
where $\wh{W}$ is the Fourier transform of $W$.

We first consider the term $h=0$ in $\Delta_{N_1,N_2}(m)$, which is equal to
\begin{align}
\frac{N_2}{\vp^*(q)}\wh{W}(0)\sum_{(n_1,q)=1}W\left(\frac{n_1}{N_1}\r)\sum_{d\mid q}\frac{\vp(d)}{d}\mu\left(\frac qd\r)\sum_{a\mid q_d}\frac{\mu(a)}{a},\notag
\end{align}
where the sums over $d$ and $a$ cancel adequately.
To see this, we alter the order of the summation, then
\begin{align}
\sum_{d\mid q}\frac{\vp(d)}{d}\mu\left(\frac qd\r)\sum_{a\mid q_d}\frac{\mu(a)}{a}=\sum_{a\mid q}\frac{\mu(a)}{a}\sum_{d\mid q_a}\frac{\vp(d)}{d}\mu\left(\frac qd\r).\notag
\end{align}
For the right-hand side, if there is a prime $p$ with $p\mid a$ and $p^2\mid q$, it vanishes as $\mu(q/d)=0$. If not the case, there is $q_a=\frac qa$, and it follows that
\[
\sum_{d\mid q_a}\frac{\vp(d)}{d}\mu\left(\frac qd\r)=\mu(a)\sum_{d\mid \frac q{a}}\frac{\vp(d)}{d}\mu\left(\frac q{ad}\r)=a\frac{\mu(q)}q
\]
 by Lemma \ref{lemsumd}.
Applying this into the sum over $a$, we find that it vanishes too.

Now we come to the contribution of the terms $h\neq0$. We have
\begin{align}
B_{M,N_1,N_2}\ll\sum_{d\mid q}\frac{\vp(d)}d\mu^2\left(\frac qd\r) \sum_{a\mid q_d}\mu^2(a)|R(d,a)|,\notag
\end{align}
where
\begin{align}\label{eqR}
R(d,a)=\frac{N_2}{a\vp^*(q)\sqrt{MN}}\sum_{(m,q)=1}d(m)W\left(\frac mM\r)\sum_{(n_1,q)=1}\sum_{h\neq0} e\left(\frac{hm\ol{a}\ol{n}_1}{d}\r) W\left(\frac {n_1}{N_1}\r)\wh{W}\left(\frac {h}{H}\r)
\end{align}
with
\[
H=adN_2^{-1}\ll qN_2^{-1}.
\]
Due to the rapid decay of the Fourier transform $\wh{W}$, contribution from such terms with $|h|\ge Hq^{\ve}$ is negligible. So we assume that the $h$-sum is over $h\le Hq^\ve$ with $Hq^\ve=adq^\ve N_2^{-1}\gg 1$, in particular,
\begin{align}\label{eqnu2}
\nu_2\le 1+\ve.
\end{align}
Now the problem reduces to showing
\begin{align}\label{eqmainRd0}
R(d,a)\ll q^{-\eta+\ve}.
\end{align}
Note that the $\ve$ here allows our neglect of other $\ve$ in following calculations.

Before further evolutions of $R(d,a)$, we first recall the region of the exponents here. We sum up from \eqref{eq1munu}, \eqref{eqnu1nu2u}, and \eqref{eqnu2} that
\begin{align}\label{eqregion}
&2-2\eta\le \nu+\mu\le 2,\ \ \ \ \ \ \ \ \ \ \ \ 1-2\theta-2\eta\le\nu-\mu,\\
&\nu=\nu_1+\nu_2,\ \ \ \ \ \ \ \ \ \ \ \ \ \ \ \ \ \ \ \ \ \ \ \ \ \tfrac12-\theta-2\eta\le\nu_1\le \nu_2\le 1,\notag
\end{align}
where $\frac12-\theta-2\eta$ in the last inequality is a direct result of other three inequalities.
We split the region into several parts, according to
\begin{enumerate}
  \item $\nu-\mu\ge 1+2\eta$;
  \item $1-2\theta-2\eta\le\nu-\mu<1+2\eta$,
  \begin{itemize}
           \item $\frac12-\theta-2\eta< \nu_1< \frac12+2\eta$;
           \item $\nu_1\ge \frac12+2\eta$.
         \end{itemize}
\end{enumerate}

\subsection{Estimate for $\nu-\mu$ large}
In this section, we prove \eqref{eqmainRd0} for $\nu-\mu\ge 1+2\eta$, which turns out to be a direct result of the Weil bound.
Applying the Weil bound to the sum over $n_1$ and bounding other sums trivially, we have
\begin{align}
R(d,a)\ll \frac{N_2H}{aq^{1+\ve}}\left(\frac NM\r)^{-\frac12}\left(d^{\frac12+\ve}+N_1d^{-1}\r),\notag
\end{align}
which after a simple calculation with $H=adN_2^{-1}$ becomes
\begin{align}
R(d,a)\ll q^{\frac12+\ve}\left(\frac NM\r)^{-\frac12}+\left(\frac NM\r)^{-\frac12}\ll q^{-\eta+\ve}\notag
\end{align}
since $N_1,d\le q^{1+\ve}$.

\subsection{Estimate for $\nu-\mu$ close to $1$}
This section is devoted to proving \eqref{eqmainRd0} for $1-2\theta-2\eta\le\nu-\mu<1+2\eta$.
We combine $m$ with $h$ to one longer variable $l=mh$. Then a trivial estimate shows
\begin{align}\label{eqRd0-DS}
R(d,a)\ll\frac{N_2q^\ve}{aq\sqrt{MN}}\sum_{l\le L}\left|\sum_{(n_1,q)=1} e\left(\frac{\ol{a}\ol{n}_1l}{d}\r) W\left(\frac {n_1}{N_1}\r)\r|
\end{align}
with
\[
 L=MHq^\ve\ll \frac{adM}{N_2}q^\ve.
\]
To bound these double sums, we apply the following lemma; see also \cite[Theorem 2.4]{KSWX22}.

\begin{lemma}\label{lemDS}
Let $q$ be a positive integer and $(\alpha_k)$ be a sequence of complex numbers satisfying $\alpha_k\ll k^\ve$. For any positive integers $L,K\le q$, we have
\begin{align}
\sum_{l\le L}\Bigg|\sum_{\substack{k\le K\\ (k,q)=1}} \alpha_k e\left(\frac{al\ol{k}}{q}\r) \Bigg|\ll LKq^\ve\cdot\Delta(L,K,q)\notag
\end{align}
uniformly in $a$ with $(a,q)=1$, where we may take the saving $\Delta(L,K,q)$ freely among
\begin{subequations}\begin{align}
&L^{-\frac{1}2}K^{-\frac{1}4}q^{\frac{1}4}+K^{-\frac{1}2},\label{eqDS1}\\
&L^{-\frac12}K^{-1}q^{\frac34}+L^{-\frac{1}2}+K^{-\frac12}.\label{eqDS2}
\end{align}
\end{subequations}
\end{lemma}
Note in \eqref{eqRd0-DS} that $L, N_1>d$ is possible, and we should extend Lemma \ref{lemDS} for all positive integers $L, K$. Since $L, K$ in \eqref{eqDS1} and \eqref{eqDS2} take value at most $q$, we may take the saving $\Delta(L,K,q)$ freely among
\begin{subequations}\begin{align}
&L^{-\frac{1}2}K^{-\frac{1}4}q^{\frac{1}4}+L^{-\frac{1}2}+q^{-\frac{1}2}+K^{-\frac{1}2},\label{eqDS3}\\
&L^{-\frac12}K^{-1}q^{\frac34}+K^{-1}q^{\frac14}+L^{-\frac{1}2}+q^{-\frac12}+K^{-\frac{1}2}\label{eqDS4}
\end{align}
\end{subequations}
instead.

{\bf Case I:}$1-2\theta-2\eta\le\nu-\mu<1+2\eta$ and $\frac12-\theta-2\eta< \nu_1< \frac12+2\eta$.

In this situation, we apply Lemma \ref{lemDS} with \eqref{eqDS3}, and after a simplification we have
\[
R(d,a)\ll\frac{N_2q^\ve}{aq\sqrt{MN}}\left(L^{\frac12}N_1^{\frac34}d^{\frac14}+L^{\frac12}N_1+L N_1 d^{-\frac12}+L N_1^{\frac12}\r),
\]
provided $N_1>q^{\frac13}$.
Recalling
\[
L\ll \frac{adM}{N_2}q^\ve,\ \ \ \ d\le q,\ \ \ \ N_1\le q,\ \ \ \  N_1N_2=N,
\]
we have
\begin{align*}
R(d,a)&\ll q^{-\frac14+\ve}N_1^{\frac14}+q^\ve\left(\frac NM\r)^{-\frac12}N_1^{\frac12}\ll  q^{-\frac14+\frac14\nu_1+\ve}+q^{-\frac12(\nu-\mu)+\frac12\nu_1+\ve}.
\end{align*}
With $\nu-\mu>1-2\theta-2\eta$ and $\nu_1< \frac12+2\eta$, it is easy to see
\[
-\tfrac14+\tfrac14\nu_1\le-\tfrac18+\tfrac12\eta\le -\eta \quad \text{for} \quad \eta\le \tfrac1{12},
\]
\[
-\tfrac12(\nu-\mu)+\tfrac12\nu_1\le -\tfrac14+\theta+2\eta\le -\eta \quad \text{for} \quad \eta\le \tfrac1{12}-\tfrac13\theta,
\]
and this establishes  \eqref{eqmainRd0}.

{\bf Case II:}$1-2\theta-2\eta\le\nu-\mu<1+2\eta$ and $\frac12+2\eta\le \nu_1$.

In this case, an easy calculation shows
\[
\tfrac12+2\eta\le \nu_1\le \tfrac12+\tfrac14(\nu-\mu)\le\tfrac34+\tfrac12\eta.
\]
Applying Lemma \ref{lemDS} with \eqref{eqDS4} into \eqref{eqRd0-DS}  gives
\[
R(d,a)\ll\frac{N_2q^\ve}{aq\sqrt{MN}}\left(L^{\frac12}d^{\frac34}+Ld^{\frac14}+L^{\frac12}N_1+LN_1d^{-\frac12}+LN_1^{\frac12}\r).
\]
By an easy calculation with
\[
L\ll \frac{adM}{N_2}q^\ve,\ \ \ \ d\le q,\ \ \ \ q^{\frac12}< N_1\le q,\ \ \ \  N_1N_2=N,
\]
it follows that
\begin{align*}
R(d,a)&\ll q^{\frac14+\ve}N_1^{-\frac12}+q^{-\frac12+\ve}N_1^{\frac12}+q^\ve\left(\frac NM\r)^{-\frac12}N_1^{\frac12}\\
&\ll q^{\frac14-\frac12\nu_1+\ve}+q^{-\frac12+\frac12\nu_1+\ve}+q^{-\frac12(\nu-\mu)+\frac12\nu_1+\ve}.
\end{align*}
Applying $\nu-\mu\ge 1-2\theta-2\eta$ and $\tfrac12+2\eta\le \nu_1\le \tfrac12+\tfrac14(\nu-\mu)\le\tfrac34+\tfrac12\eta$, we have
\[
\tfrac14-\tfrac12\nu_1\le \tfrac14-\tfrac12\times(\tfrac12+2\eta)= -\eta,
\]
\[
-\tfrac12+\tfrac12\nu_1\le -\tfrac18+\tfrac14\eta\le -\eta  \quad \text{for} \quad \eta\le \tfrac1{10},
\]
\[
-\tfrac12(\nu-\mu)+\tfrac12\nu_1\le -\tfrac38(\nu-\mu)+\tfrac14\le -\tfrac18+\tfrac34\theta+\tfrac34\eta\le -\eta \quad \text{for} \quad \eta\le \tfrac1{14}-\tfrac37\theta,
\]
and the result follows.

\noindent
\section*{Acknowledgments}
The author would like to take the opportunity to thank Brian Conrey, Roger Heath-Brown, Henryk Iwaniec, and Matthew Young for their encouragement on this project. The author also wants to thank Lilu Zhao for helpful discussions. This work is supported in part by the National Natural Science Foundation of China (Grant No. 11871187) and the Fundamental Research Funds for the Central Universities of China.

\end{document}